\documentclass[a4paper]{amsart}
\usepackage{amssymb}
\usepackage{amsmath}
\usepackage{xypic}

\newtheorem{thm}{Theorem}[section]
\newtheorem*{thm*}{Theorem}

\newtheorem{cor}[thm]{Corollary}
\newtheorem*{cor*}{Corollary}
\newtheorem{lemma}[thm]{Lemma}
\newtheorem{prop}[thm]{Proposition}
\theoremstyle{definition}
\newtheorem{defn}[thm]{Definition}

\theoremstyle{remark}
\newtheorem{rem}[thm]{Remark}

\newcommand{\Ann}[1]{{#1}^0}
\newcommand{\hook}{\lrcorner \,}

\newcommand{\ZZ}{{\mathbb  Z}}
\newcommand{\RR}{{\mathbb  R}}
\newcommand{\HH}{{\mathbb  H}}
\newcommand{\CC}{{\mathbb  C}}

\newcommand{\NN}{{\mathbb  N}}

\newcommand{\G}{\Gamma}

\newcommand{\su}{{\mathfrak{su}}} 
\newcommand{\msl}{{\mathfrak{sl}}} 
\newcommand{\SU}{{\mathrm{SU}}}

\renewcommand{\G}{{\mathrm{G}}}

\newcommand{\mfg}{\mathfrak{g}}
\newcommand{\mfh}{\mathfrak{h}}

\newcommand{\mfu}{\mathfrak{u}}

\newcommand{\mfp}{\mathfrak{p}}

\newcommand{\ad}{\operatorname{ad}}

\newcommand{\id}{\operatorname{id}}

\newcommand{\SO}{\operatorname{SO}}
\newcommand{\GL}{\operatorname{GL}}
\newcommand{\SL}{\operatorname{SL}}
\newcommand{\Spin}{\operatorname{Spin}}

\newcommand{\spa}[1]{\mathrm{span}(#1)}
\def\bi{\begin{enumerate}}
\def\ei{\end{enumerate}}

\numberwithin{equation}{section}

\begin{document}

\title{On the boundary behaviour of left-invariant Hitchin and hypo flows}

\author{Florin Belgun}
\author{Vicente Cort\'es}
\author{Marco Freibert}
\author{Oliver Goertsches}
\address{Florin Belgun, Fachbereich Mathematik, Universit\"at Hamburg, Bundesstra\ss e 55, 20146 Hamburg, Germany}
\address{Vicente Cort\'es, 
Fachbereich Mathematik und Zentrum f\"ur Mathematische Physik, 
Universit\"at Hamburg, Bundesstra\ss e 55, 20146 Hamburg, Germany}
\address{Marco Freibert, 
Department of Mathematics, Aarhus University, Ny Munkegade 118, Bldg 1530, DK-8000 Aarhus C, Denmark}
\address{Oliver Goertsches, 
Fachbereich Mathematik, Universit\"at Hamburg, Bundesstra\ss e 55, 20146 Hamburg, Germany}

\date{}


\begin{abstract}
We investigate left-invariant Hitchin and hypo flows on $5$-, $6$- and $7$-dimensional Lie groups. They provide Riemannian cohomogeneity-one manifolds of one dimension higher with holonomy contained in $\SU(3)$, $\G_2$ and $\Spin(7)$, respectively, which are in general geodesically incomplete. Generalizing results of Conti, we prove that for large classes of solvable Lie groups $G$ these manifolds cannot be completed: a complete Riemannian manifold with parallel $\SU(3)$-, $\G_2$- or $\Spin(7)$-structure which is of cohomogeneity one with respect to $G$ is flat, and has no singular orbits.

We furthermore classify, on the non-compact Lie group $\SL(2,\CC)$, all half-flat $\SU(3)$-structures which are bi-invariant with respect to the maximal compact subgroup $\SU(2)$ and solve the Hitchin flow for these initial values. It turns out that often the flow collapses to a smooth manifold in one direction. In this way we recover an incomplete cohomogeneity-one Riemannian metric with holonomy equal to $\G_2$ on the twisted product $\SL(2,\CC)\times_{\SU(2)} \CC^2$ described by Bryant and Salamon.\end{abstract}

\maketitle
\section{Introduction}
The Hitchin flow \cite{Hitchin} starts with a half-flat
$\SU(3)$-structure on a $6$-dimensional or a cocalibrated
$\G_2$-structure on a $7$-dimensional manifold $M$, and constructs
from this initial data a parallel $\G_2$- or $\Spin(7)$-structure,
respectively, on the product $M\times I$ of $M$ with an interval
$I$. An analogue of this flow for hypo $\SU(2)$-structures was
introduced in \cite{CS}, resulting in a parallel
$\SU(3)$-structure. If the interval on which the flow is defined is
not the whole real line, then the resulting Riemannian manifold is
geodesically incomplete. One would like to find conditions under which
the flow degenerates in a controlled way at the boundaries of the
interval, in order to obtain a natural metric completion, which then
still carries the same geometry. A natural simplifying assumption in this context is to require homogeneity of the initial data, so that $M\times I$ is of cohomogeneity one with only regular orbits. Under mild assumptions, the cohomogeneity-one action extends automatically to any potential completion of $M\times I$, see Proposition \ref{prop:extensionprop}.

For initial data which is homogeneous under a compact Lie group these flows and the above extension problem were studied extensively in the literature, see, e.g., \cite{Brand}, \cite{CleytonSwann}, \cite{Madsen}, \cite{Reidegeld}, \cite{Reidegeld2}. In this paper we focus on the case of left-invariant initial data on a possibly non-compact Lie group, a setting which was previously considered in \cite{Chong}, \cite{Conti2} and \cite{CLSS}. It was shown by Conti \cite[Section 8]{Conti2} that this problem has only trivial solutions for the hypo flow on nilpotent Lie groups, in the sense that the resulting manifold is automatically flat. His proof uses his classification of hypo $\SU(2)$-structures on $5$-dimensional nilpotent Lie groups. We find a new, conceptional proof of his statement, which works more generally for arbitrary split solvable Lie groups $G$, see Theorem \ref{thm:flatness} (a). With analogous arguments, we are able to show that also the Hitchin flow on certain classes of six- and seven-dimensional split-solvable Lie groups $G$ can only be extended trivially, see part (b) and (c) of the same theorem. To prove this statement, we first show in Section \ref{sec:singorbits}, as Conti did for the hypo flow on nilpotent Lie groups, that in all considered cases one cannot extend $G\times I$ to a, not necessarily complete, Riemannian manifold of cohomogeneity one with respect to $G$ with one or more singular orbits.

Examples of Lie groups for which the hypo or Hitchin flow extends non-trivially are rarely known. The by far most studied case is $S^3\times S^3$ on which the solution of the extension problem led to several complete Riemannian manifolds, cf., e.g., \cite{Brand}, \cite{Chong}, \cite{Madsen}. Note that even the first example of a complete Riemannian manifold with holonomy equal to $\G_2$ on the spin bundle over $S^3$ \cite{BryantSalamon} is an extension in the above sense of the Riemannian manifold obtained by a Hitchin flow with left-invariant initial $\SU(3)$-structure \cite{Hitchin}. In this case, the initial value is biinvariant under the diagonal $\SU(2)$, which simplifies the flow equations greatly.

To obtain a similar example on a non-compact Lie group, we consider the Lie group $\SL(2,\CC)$, whose Lie algebra has the same complexification as $\mathfrak{su}(2)\oplus \mathfrak{su}(2)$, and impose invariance of the initial value under the maximal compact subgroup $\SU(2)$. We classify in Section \ref{sec:example} all $\SU(2)$-invariant left-invariant half-flat $\SU(3)$-structures on $\SL(2,\CC)$ and solve the Hitchin flow explicitly for all these initial values. The initial values depend on three continuous and one discrete parameter and the solutions of the Hitchin flow are defined on a finite interval $(a,b)$ for all possible parameters. Whereas at one boundary point the solution always collapses in a bad way such that one cannot extend the Riemannian manifold in that direction, the degeneration at the other end behaves nicely precisely when one specific continous parameter vanishes. In these cases, we can extend the corresponding Riemannian manifolds in one direction and obtain incomplete cohomogeneity-one Riemannian metrics on the twisted product $\SL(2,\CC)\times_{\SU(2)} \CC^2$ which all turn out to be homothetic to a metric with holonomy equal to $\G_2$ on the spin bundle over three-dimensional hyperbolic space described by Bryant and Salamon \cite[Section 3]{BryantSalamon}.\\

{\em Acknowledgements:} We are grateful to Frank Reidegeld for helpful discussions. This work was supported by the Collaborative Research Center SFB 676 ``Particles, Strings, and the Early Universe'' of the Deutsche Forschungsgemeinschaft. The third author was also partly supported by Det Frie Forskningsr\aa d through the DFF-Individual postdoctoral grant DFF -- 4002-00125.

\section{Flow equations and special holonomy}
In this section we give a brief overview of $\SU(2)$-, $\SU(3)$-, $\G_2$- and $\Spin(7)$-structures in dimension five, six, seven and eight, respectively, and their relation to the special holonomy groups $\SU(3)$, $\G_2$ and $\Spin(7)$ in six, seven and eight dimensions, respectively, via certain flow equations. For more details on \linebreak $\SU(3)$-, $\G_2$- and $\Spin(7)$-structures and proofs of the mentioned facts, the reader may consult, e.g., \cite{CLSS}, \cite{Hitchin} and \cite{Hitchin2}. For $\SU(2)$-structures, the main references are \cite{Conti2} and \cite{CS}. Note that in \cite{Stock} a unified treatment of all cases is given.

We begin with the definition of the mentioned $G$-structures:
\begin{defn}
\begin{itemize}
\item
An \emph{$\SU(2)$-structure} on a five-dimensional manifold $M$ is a quadruple $(\alpha,\omega_1,\omega_2,\omega_3)\in \Omega^1 M\times (\Omega^2 M)^3$ for which at each point $p\in M$ there exists an ordered basis $(e_1,\ldots,e_5)$ of $T_p M$ with
\begin{equation*}
\alpha_p=e^5,\quad (\omega_1)_p=e^{12}+e^{34},\quad (\omega_2)_p=e^{13}-e^{24},\quad (\omega_3)_p=e^{14}+e^{23}.
\end{equation*}
  The automorphism group of the above defined structure (i.e., the
  group of transformations preserving the forms $\alpha, \omega_1,
  \omega_2,  \omega_3$) is $\SU(2)\subset \SO(5)$.
$(\alpha,\omega_1,\omega_2,\omega_3)$ is called \emph{hypo} if
\begin{equation*}
d\;\!\omega_1=0,\quad d(\alpha\wedge \omega_2)=0,\quad d(\alpha\wedge \omega_3)=0.
\end{equation*}
\item
An \emph{$\SU(3)$-structure} on a six-dimensional manifold $M$ is a pair $(\omega,\rho)\in \Omega^2 M\times \Omega^3 M$ for which at each point $p\in M$ there exists an ordered basis $(e_1,\ldots,e_6)$ of $T_p M$ with
\begin{equation*}
\omega_p=e^{12}+e^{34}+e^{56},\quad \rho_p=e^{135}-e^{146}-e^{236}-e^{245}.
\end{equation*}
  The automorphism group of this structure (i.e., the
  group of transformations preserving the forms $\omega,\rho$) is $\SU(3)\subset \SO(6)$.
$(\omega,\rho)$ is called \emph{half-flat} if
\begin{equation*}
d\left(\omega^2\right)=0,\quad d\:\!\rho=0.
\end{equation*}
\item
A \emph{$\G_2$-structure} on a seven-dimensional manifold $M$ is a three-form $\varphi\in \Omega^3 M$ for which at each point $p\in M$ there exists an ordered basis $(e_1,\ldots,e_7)$ of $T_p M$ with
\begin{equation*}
\varphi_p=e^{127}+e^{347}+e^{567}+e^{135}-e^{146}-e^{236}-e^{245}.
\end{equation*}
 The automorphism group of the structure (i.e., the
  group of transformations preserving the form $\varphi$) is
  $\G_2\subset \SO(7)$.
\item
A \emph{$\Spin(7)$-structure} on an eight-dimensional manifold $M$ is a four-form $\Phi\in \Omega^4 M$ for which at each point $p\in M$ there exists an ordered basis $(e_1,\ldots,e_8)$ of $T_p M$ with
\begin{equation*}
\begin{split}
\Phi_p=\ &e^{1278}+e^{3478}+e^{5678}+e^{1358}-e^{1468}-e^{2368}-e^{2458}\\
&+ e^{1234}+e^{1256}+e^{3456}+e^{1367}+e^{1457}+e^{2357}-e^{2467}.
\end{split}
\end{equation*}
 The automorphism group of the structure (i.e., the
  group of transformations preserving the form $\Phi$) is
  $\Spin(7)\subset \SO(8)$.
\item
An ordered basis $(e_1,\ldots,e_n)$ of $T_p M$ as in the definition of the different $G$-structures above is called an \emph{adapted} basis. As $\SU(2)\subseteq \SO(5)$, $\SU(3)\subseteq \SO(6)$, $\G_2\subseteq \SO(7)$ and $\Spin(7)\subseteq \SO(8)$, each of the above considered $G$-structures induces a Riemannian metric $g$ and an orientation and so also a Hodge star operator. The Riemannian metric and the orientation are defined by the property that any adapted basis is an oriented orthonormal basis. The $G$-structure is called \emph{parallel} if the defining differential forms are parallel with respect to the associated Levi-Civita connection $\nabla^g$. Note that the holonomy of $g$ is then contained in $G$.
\item
For a $\G_2$-structure $\varphi$ one gets
\begin{equation}\label{eq:Hodgedual}
(\star_{\varphi}\varphi)_p=e^{1234}+e^{1256}+e^{3456}+e^{1367}+e^{1457}+e^{2357}-e^{2467}.
\end{equation}
for any adapted basis $(e_1,\ldots,e_7)$ at $p\in M$. We call $\varphi$ \emph{cocalibrated} if 
\begin{equation*}
d\star_{\varphi}\varphi=0.
\end{equation*}
\end{itemize}
\end{defn}
It is possible to give an ``intrinsic'' definition of all these $G$-structures, i.e., a definition without referring to an adapted basis but instead through certain properties the defining differential forms have to fulfill. We elaborate this in more detail only in the $\SU(3)$-case since we will not need the intrinsic definition in the other cases. For this, we first need some prepatory definitions.
\begin{defn}
Let $\rho\in \Lambda^3 V^*$ be a three-form on an oriented
 six-dimensional vector space $V$. Define a linear map $K_{\rho}:V\rightarrow V\otimes \Lambda^6 V^*$ by

\begin{equation}\label{eq:Krho}
K_{\rho}(v)=\kappa(v\hook \rho\wedge \rho)
\end{equation}

for all $v\in V$, where $\kappa: \Lambda^5 V^*\rightarrow  V \otimes\Lambda^6 V^*$ is the natural $\GL(V)$-equivariant isomorphism. Set now 
\begin{equation}\label{eq:lambda}
\lambda(\rho):=\frac{1}{6}\mathrm{tr}((K_{\rho}\otimes \id_{\Lambda^6 V^*})\circ K_{\rho} )\in \left(\Lambda^6 V^*\right)^{\otimes 2}.
\end{equation}

If $\lambda(\rho)<0$, we define  the square root of $-\lambda(\rho)$ as
  the unique positive (for the given orientation) 6-form on $V$ that
  squares to $-\lambda(\rho)\in(\Lambda^6 V^*)^{\otimes 2}$ and the
endomorphism $J_{\rho}$ of $V$ by 
\begin{equation}\label{eq:Jrho}
K_{\rho}= J_{\rho}\otimes \sqrt{-\lambda(\rho)}
\end{equation}
It is well-known that $J_{\rho}$ is a complex structure on the vector space $V$ \cite{Hitchin2}. Furthermore, we set

\begin{equation}\label{eq:rhohat}
\hat{\rho}:=J_{\rho}^*\rho\in \Lambda^3 V^*
\end{equation}

and
\begin{equation}\label{eq:complexvolume}
\Psi:=\rho+i\hat{\rho}\in \Lambda^3 V^*\otimes \CC.
\end{equation}
\end{defn}
One obtains now the following characterization of $\SU(3)$-structures.
\begin{lemma}\label{lem:characterizationSU3}
Let $M$ be a six-dimensional manifold. Then a pair $(\omega,\rho)$ of
a two-form $\omega$ and a three-form $\rho$ is an $\SU(3)$-structure
if and only if 
\bi
\item[(a)] $\omega$ is non-degenerate (consider then the orientation of
  $TM$ given by the volume form
  $\omega^3$), 
\item[(b)] $\lambda(\rho)<0$, 
\item[(c)] $\omega\wedge \rho=0$,
\item[(d)] $\sqrt{-\lambda(\rho)}=\frac{\omega^3}{3}$ and 
\item[(e)] $g_{(\omega,\rho)}:=\omega(J_{\rho}\cdot,\cdot)$ is a
  Riemannian metric.
\ei
If this is the case, then $(g,J_{\rho},\omega)$ is an almost Hermitian structure and $\Psi$ is a \emph{complex volume form}, i.e. a non-zero complex $(3,0)$-form, on $M$.
\end{lemma}
\begin{rem}
We would like to note that $\SU(3)$-structures in our sense are sometimes called \emph{normalized} $\SU(3)$-structures, cf., e.g., \cite{CLSS}. A (non-normalized) $\SU(3)$-structure is then a pair $(\omega,\rho)\in \Omega^2 M\times \Omega^3 M$ fulfilling the conditions of Lemma \ref{lem:characterizationSU3} but instead of $\sqrt{-\lambda(\rho)}=\frac{\omega^3}{3}$, only $\sqrt{-\lambda(\rho)}=c\frac{\omega^3}{3}$ for some non-zero constant $c\in\RR$.
\end{rem}

\begin{rem} $\lambda(\rho)$ is an element of an oriented
  representation of $\GL(V)$, hence the sign of $\lambda(\rho)$ (and thus
  the condition (b) above) is
  independent of the chosen orientation of $V$. Moreover, note that
  the condition (c) implies that the tensor
  $g_{(\omega,\rho)}$ in the Lemma above is symmetric (and
  non-degenerate because of condition (a)).
\end{rem}

Next, we would like to note some algebraic properties of $\SU(2)$-, $\SU(3)$- and $\G_2$-structures which we will need in Section \ref{sec:singorbits}.
\begin{lemma}\label{lem:algprop}
\begin{enumerate}
\item[(a)]
Let $(\alpha,\omega_1,\omega_2,\omega_3)$ be an $\SU(2)$-structure on a five-dimensio\-nal manifold $M$. Then the kernels of the forms $\omega_1$, $\omega_2$ and $\omega_3$ at each point $p\in M$ are equal and one-dimensional. Moreover, $\alpha$ is non-zero when restricted to this common kernel.
\item[(b)]
Let $(\omega,\rho)$ be an $\SU(3)$-structure on a six-dimensional manifold $M$. If at some point $p\in M$ there are tangent vectors $X$ and $Y$ with $\rho_p(X,Y,\cdot)=0$, then $X$ and $Y$ are $(J_{\rho})_p$-complex linearly dependent and $\hat{\rho}_p(X,Y,\cdot)=0$.
\item[(c)]
Let $\varphi$ be a $\G_2$-structure on a seven-dimensional manifold $M$. For all $p\in M$ and all linearly independent $X,\, Y\in T_p M$, the kernel of the two-form $(\star_{\varphi} \varphi)_p(X,Y,\cdot,\cdot)$ is three-dimensional.
\end{enumerate}
\end{lemma}
\begin{proof}
\begin{enumerate}
\item[(a)]
Can directly be deduced from the definition.
\item[(b)]
Let $X,\,Y\in T_p M$ with $\rho_p(X,Y,\cdot)=0$. Assume now that $X$ and $Y$ are $(J_{\rho})_p$-complex linearly independent. Choose $Z\in T_p M$ such that $\{X,\, Y,\, Z\}$ is a $(J_{\rho})_p$-complex basis of $T_p M$. As $\Psi=\rho+i\hat{\rho}$ is a complex $(3,0)$-form, $\Psi_p(X,Y,Z)\neq 0$ and so we must have $\hat{\rho}_p(X,Y,Z)\neq 0$ as $\rho_p(X,Y,Z)=0$. But then
\begin{equation*}
\begin{split}
\rho_p(X,Y,JZ)=&\Psi_p(X,Y,JZ)-i\hat{\rho}_p(X,Y,JZ)=i\Psi_p(X,Y,Z)-i\hat{\rho}_p(X,Y,JZ)\\
=&-\hat{\rho}_p(X,Y,Z)-i\hat{\rho}_p(X,Y,JZ)\neq 0,
\end{split}
\end{equation*}
a contradiction. Hence $X$ and $Y$ must be $(J_{\rho})_p$-complex linearly dependent. Thus, $\Psi_p(X,Y,\cdot)=0$ and so also $\hat{\rho}_p(X,Y,\cdot)=0$.
\item[(c)]
This follows directly from \cite[Lemma 2.24]{CHNP}.
\end{enumerate}
\end{proof}
From the definition of the considered $G$-structures in terms of adapted bases one can already guess that there are certain connections between the $G$-structures in dimension $n$ and $n+1$. More exactly, one gets an induced hypo $\SU(2)$-structure on any oriented hypersurface in a six-dimensional manifold with a parallel $\SU(3)$-structure, a half-flat $\SU(3)$-structure on any oriented hypersurface in a seven-dimensional manifold with a parallel $\G_2$-structure and a cocalibrated $\G_2$-structure on any oriented hypersurface in an eight-dimensional manifold with a parallel $\Spin(7)$-structure. If one considers now an equidistant family of oriented hypersurfaces, then the induced smooth one-parameter families of hypo $\SU(2)$-, half-flat $\SU(3)$- or cocalibrated $\G_2$-structures fulfill certain time-dependent partial differential equations, which are called \emph{hypo flow equations} in the $\SU(2)$-case and \emph{Hitchin's flow equations} in the other two cases. Conversely, one can start with smooth one-parameter families of these structures on a manifold $M$ fulfilling the corresponding flow equations for all $t$ in an interval $I$ and construct on $M\times I$ the corresponding parallel $G$-structure. More exactly, for hypo $\SU(2)$-structures and parallel $\SU(3)$-structures, the following is true, cf. \cite{CS}:
\begin{prop}\label{prop:hypoflow}
Let $M$ be a five-dimensional manifold, $I$ be an open interval and $t$ be the standard coordinate on $I$. Moreover, let $(\omega,\rho)\in \Omega^2 (M\times I)\times \Omega^3 (M\times I)$ be a parallel $\SU(3)$-structure on $M\times I$ such that the induced Riemannian metric on $M\times I$ is of the form $g_t+dt^2$. Then the induced smooth one-parameter family $I\ni t\mapsto (\alpha_t,(\omega_1)_t,(\omega_2)_t,(\omega_3)_t)\in \Omega^1 M\times (\Omega^2 M)^3$ defined by
$\alpha_t=-\left.\frac{\partial}{\partial t}\hook \omega\right|_{M\times \{t\}}$, $(\omega_1)_t=\left. \omega\right|_{M\times \{t\}}$,
$(\omega_2)_t=-\left.\frac{\partial}{\partial t}\hook \hat{\rho}\right|_{M\times \{t\}}$ and $(\omega_3)_t=-\left.\frac{\partial}{\partial t}\hook \rho\right|_{M\times \{t\}}$ consists of hypo $\SU(2)$-structures fulfilling the \emph{hypo flow equations}
\begin{equation}\label{eq:hypoflow}
\frac{d}{dt}\left(\omega_1\right)_t=-d\alpha_t,\quad \frac{d}{dt}\left(\left(\omega_2\right)_t\wedge \alpha_t\right)=-d\left(\omega_3\right)_t, \quad \frac{d}{dt}\left(\left(\omega_3\right)_t\wedge \alpha_t\right)=d\left(\omega_2\right)_t
 \end{equation}
for all $t\in I$. Conversely, any smooth one-parameter family of
$\SU(2)$-structures $I \ni t\mapsto
(\alpha_t,(\omega_1)_t,(\omega_2)_t,(\omega_3)_t)$ on $M$ which is
hypo for some $t_0\in I$ and fulfills the hypo flow equations
\eqref{eq:hypoflow} on $I$ is a family of hypo $\SU(2)$
  structures, with corresponding Riemannian metrics 
$$g_t:=g_{\left(\alpha_t,(\omega_1)_t,(\omega_2)_t,(\omega_3)_t\right)}$$
  and it defines a parallel $\SU(3)$-structure $(\omega,\rho)\in \Omega^2 (M\times I)\times \Omega^3 (M\times I)$ via
\begin{equation}\label{eq:parallelSU3}
\omega:=(\omega_1)_t+\alpha_t\wedge dt,\quad \rho:=(\omega_2)_t\wedge \alpha_t-(\omega_3)_t\wedge dt.
\end{equation}
The induced Riemannian metric $g_{(\omega,\rho)}$ on $M\times I$ is given by 
\begin{equation}
g_{(\omega,\rho)}=  g_t+dt^2
\end{equation} and has holonomy in $\SU(3)$.
\end{prop}
By \cite{Hitchin} and \cite{CLSS}, one gets the following proposition for half-flat $\SU(3)$-structures and parallel $\G_2$-structures:
\begin{prop}\label{prop:Hitchinflow}
Let $M$ be a six-dimensional manifold, $I$ be an open interval and $t$ be the standard coordinate on $I$. Moreover, let $\varphi\in \Omega^3 M$ be a parallel $\G_2$-structure on $M\times I$ such that the induced Riemannian metric is of the form $g_t+dt^2$. Then the induced smooth one-parameter family $I\ni t\mapsto (\omega_t,\rho_t)\in \Omega^2 M\times \Omega^3 M$ given by
$\omega_t:=\left.\frac{\partial}{\partial t}\hook \varphi\right|_{M\times \{t\}}$ and $\rho_t:=\varphi|_{M\times \{t\}}$ consists of half-flat $\SU(3)$-structures which fulfill Hitchin's flow equations
\begin{equation}\label{eq:Hitchinflow}
\frac{d}{dt}\rho_t=d\omega_t,\quad \frac{d}{dt}\left(\frac{\omega_t^2}{2}\right)=d\widehat{\rho_t}.
 \end{equation}
Conversely, any smooth one-parameter family $I\ni t\mapsto (\omega_t,\rho_t)$ of $\SU(3)$-structures on $M$ which is half-flat for some $t_0\in I$ and fulfills Hitchin's flow equations \eqref{eq:Hitchinflow} on $I$ defines a parallel $\G_2$-structure $\varphi$ on $M\times I$ given by
\begin{equation}\label{eq:parallelG2}
\varphi:=\omega_t\wedge dt+\rho_t.
\end{equation}
The Hodge dual $\star_{\varphi}\varphi$ is given by $\star_{\varphi}\varphi=\tfrac{\omega_t^2}{2}-\widehat{\rho_t}\wedge dt$ and the induced Riemannian metric $g_{\varphi}$ on $M\times I$ by $g_{\varphi}=g_{(\omega_t,\rho_t)}+dt^2$. $g_{\varphi}$ has holonomy contained in $\G_2$.
\end{prop}
For cocalibrated $\G_2$-structures and parallel $\Spin(7)$-structures, \cite{Hitchin} and \cite{CLSS} yield
\begin{prop}\label{prop:Hitchinflow2}
Let $M$ be a seven-dimensional manifold, $I$ be an open interval and $t$ be the standard coordinate on $I$. Moreover, let $\Phi\in \Omega^4 M$ be a parallel $\Spin(7)$-structure on $M\times I$ such that the induced Riemannian metric is of the form $g_t+dt^2$. Then the induced smooth one-parameter family $I\ni t\mapsto \varphi_t\in \Omega^3 M$ given by $\varphi_t:=\left.\frac{\partial}{\partial t}\hook \Phi\right|_{M\times \{t\}}$ consists of cocalibrated $\G_2$-structures which fulfill Hitchin's flow equations
\begin{equation}\label{eq:Hitchinflow2}
\frac{d}{dt}\star_{\varphi_t}\varphi_t=d\varphi_t.
 \end{equation}
Conversely, any smooth one-parameter family $I\ni t\mapsto \varphi_t$ of $\G_2$-structures on $M$ which is cocalibrated for some $t_0\in I$ and fulfills Hitchin's flow equations \eqref{eq:Hitchinflow2} on $I$ defines a parallel $\Spin(7)$-structure $\Phi$ on $M\times I$ given by
\begin{equation}\label{eq:parallelSpin7}
\Phi:=dt\wedge \varphi_t+\star_{\varphi_t}\varphi_t.
\end{equation}
The Riemannian metric $g_{\Phi}$ on $M\times I$ induced by $\Phi$ is given by $g_{\Phi}=g_{\varphi_t}+dt^2$ and has holonomy in $\Spin(7)$.
\end{prop}

\begin{rem}\label{rem:invariant}
\begin{itemize}
\item
By the Cauchy-Kowalevskaya Theorem, the flow equations \eqref{eq:hypoflow}, \eqref{eq:Hitchinflow}, \eqref{eq:Hitchinflow2} on a manifold $M$ together with an initial value at $t=t_0$ admit a unique solution on an open neighborhood $U$ of $M\times \{t_0\}$ in $M\times \RR$ provided $M$ and all initial data are real-analytic. If $M$ and the initial data are even homogeneous, the flow equations become ordinary differential equations and $U$ is of the form $M\times I$ for some open interval $I$ containing $t_0$. Moreover, $(M\times I,g)$ is a Riemannian manifold of cohomogeneity one with only principal orbits.
\item
Any diffeomorphism of $M$ which leaves the initial value of the flow equations \eqref{eq:hypoflow}, \eqref{eq:Hitchinflow}, \eqref{eq:Hitchinflow2} invariant, automatically leaves the solution invariant at any time $t\in I$.
\end{itemize}
\end{rem}
Finally, we would like to note the following immediate consequence of Proposition \ref{prop:Hitchinflow} and Proposition \ref{prop:Hitchinflow2}.
\begin{lemma}\label{lem:6to7to8}
Let $M$ be an $n$-dimensional manifold and $G$ be a Lie group acting on $M$. Moreover, let $t$ be the standard coordinate in $\RR$ and let $G\times \RR$ act in the obvious way on $M\times \RR$.
\begin{enumerate}
\item[(a)]
Let $n=6$ and $(\omega,\rho)$ be a parallel $G$-invariant $\SU(3)$-structure on $M$. Then $\varphi:=\omega\wedge dt+\rho$ is a parallel $(\G\times \RR)$-invariant $\G_2$-structure on $M\times \RR$ .
\item[(b)]
Let $n=7$ and $\varphi$ be a parallel $G$-invariant $\G_2$-structure on $M$. Then \linebreak$\Phi:=dt\wedge \varphi+\star_{\varphi}\varphi$ is a parallel $(G\times \RR)$-invariant $Spin(7)$-structure on $M\times \RR$.
\end{enumerate}
\end{lemma}
\section{Proper Lie group actions}
In this section we recall some basic properties of proper isometric Lie group actions on Riemannian manifolds. Throughout, we will assume that $G$ is a connected, but possibly non-compact Lie group. An action $G\times M\to M;\, (g,p)\mapsto g\cdot p$ of $G$ on a manifold $M$ is called proper if the map
\[
G\times M\longrightarrow M\times M;\, (g,p)\longmapsto (g\cdot p,p)
\]
is proper. As shown by Palais \cite{Palais}, many results on actions
of compact Lie groups are still valid for proper actions, see also
\cite{DuistermaatKolk}, Chapter 2, or \cite{GGK}, Appendix B. Most
importantly for us, the slice theorem (see, e.g., \cite[Theorem
B.24]{GGK}) holds true: Let $p$ be a point in a Riemannian manifold
$M$ on which a Lie group $G$ acts properly and isometrically  and
  let $G_p$ be the isotropy subgroup of $p$. Let $B$ be a small open
ball in the normal space $\nu_p(G\cdot p)$ around the origin, and form
the twisted product $G\times_{G_p}B$ (for the induced right action of
$G_p$ on $G$ and the isotropy representation of $G_p$ on $B$), which
carries a natural $G$-action by left multiplication on the first
factor. Then the slice theorem states that there exists a
$G$-equivariant diffeomorphism $\psi: G\times_{G_p} B \to U$ onto an
open $G$-invariant neighborhood $U$ of the orbit $G\cdot p$ such that
$\psi([e,0])=p$. 

We will call an orbit regular if it is of maximal dimension among all
orbits, and otherwise singular. Then the cohomogeneity of a Lie group
action is the codimension of a regular orbit. The actions we encounter
in this paper will all be of cohomogeneity one; note that the slice
theorem implies that for a proper isometric cohomogeneity-one action
any isotropy group $G_p$ of a point $p$ in a singular orbit acts transitively on the unit sphere in the normal space $\nu_p(G\cdot p)$.

The Lie groups we consider in Sections \ref{sec:singorbits} and \ref{sec:flatness} will all be solvable; for such Lie groups more restrictive statements about the isotropy groups are valid:

\begin{lemma}\label{lem:dimsingorbit} Let $G$ be a solvable Lie group acting properly and isometrically on an $n$-dimensional Riemannian manifold. Then every identity component of an isotropy group is either trivial or a torus. If the action is additionally of cohomogeneity one, then every singular orbit of the $G$-action has dimension $n-2$.
\end{lemma}
\begin{proof}
Because the action is proper, all isotropy groups are compact. Then the identity component of any isotropy is, as a compact connected subgroup of a solvable Lie group, a torus (because the Lie algebra of a compact Lie group is reductive). Assuming that the action is of cohomogeneity one, the observation before the lemma implies that any nontrivial isotropy group acts transitively on the unit sphere in the normal space. But the only sphere on which a torus can act transitively is the circle.
\end{proof}

\begin{lemma}\label{lem:isocenter}
Let $ G$ be a solvable Lie group, $H$ be a compact subgroup and denote by $\mathfrak{z}(\mfg)$ the center of the Lie algebra $\mfg$ associated to $G$. Then:
\begin{enumerate}
\item[(i)] $\mfh \cap [\mfg,\mfg]\subseteq \mathfrak{z}(\mfg)$.
\item[(ii)] 
If $ G$ is split-solvable, then $\mfh\subseteq \mathfrak{z}(\mfg)$.
\end{enumerate} 
In particular, if $G$ acts properly and isometrically on a Riemannian manifold $M$ and $\mfg_p$ is the stabilizer subalgebra of a point $p\in M$, then $\mfg_p\cap [\mfg,\mfg]\subseteq \mathfrak{z}(\mfg)$ and if $G$ is even split-solvable, then $\mfg_p\subseteq \mathfrak{z}(\mfg)$.
\end{lemma}
\begin{proof}
As a compact connected subgroup of a solvable Lie group, the identity component of $H$ is a torus, and hence $\mfh$ is Abelian. Moreover, because $H$ is compact, there exists an $H$-invariant inner product $\langle\cdot,\cdot\rangle$ on $\mfg$. Let $X\in \mfh$. Then the orthogonal complement $V=\mfh^\perp$ is invariant under $\ad_X$ and $\ad_X$ is skew-symmetric on $V$. If $X\in [\mfg,\mfg]$, $\ad_X$ is a nilpotent endomorphism as $[\mfg,\mfg]$ is nilpotent and so $\ad_X:V\to V$ has to be the zero endomorphism. If $G$ is split solvable, all eigenvalues of $\ad_X$ are real and again $\ad_X:V\to V$ equals the zero endomorphism. So, in both cases, $X$ is contained in the center of $\mfg$ as $\mfh$ was already shown to be Abelian. 
\end{proof}

\begin{rem} Notice that a stabilizer of a proper isometric action of a solvable
Lie group is not necessarily central.  Consider for example the action of the special Euclidean group $\RR^2\rtimes \SO(2)$ on $\RR^2$. 
\end{rem}

Finally we prove a proposition on the extension of a proper isometric Lie group action on an open and dense subset of a complete Riemannian manifold $M$ to all of $M$.

\begin{prop}\label{prop:extensionprop} Let $M$ be a complete Riemannian manifold, and $U$ be an open and dense subset which satisfies the property that the Riemannian distance on $M$, restricted to $U$, coincides with the Riemannian distance of $U$. Let $ G$ be a Lie group acting properly and isometrically on $U$. Then the $ G$-action extends in a unique way to an isometric $ G$-action on $M$, and the extended action is again proper.
\end{prop}
\begin{proof} 
We first show that any isometry of $U$ can be extended in a unique way
to an isometry of $M$. The uniqueness is clear as $U$ is dense
  in $M$.

So let $\varphi\in I(U)$ be an isometry of $U$. Let $p\in M$. As $U$ is dense in $M$, we can find a sequence $p_n$ in $U$ converging to $p$. Since $d_U$ equals the restriction of $d_M$ to $U$ and $\varphi$ is an isometry of $U$, the sequence $(\varphi(p_n))_n$ is a Cauchy sequence in $M$ and so has a limit $q\in M$ because $M$ is complete. We set $\varphi(p):=q$. This gives a well-defined map $\varphi:M\to M$, because for any other sequence $p_n'$ in $U$ converging to $p$, we have $d_U(\varphi(p_n),\varphi(p_n'))=d_U(p_n,p_n')\to 0$, again by the assumption on the two metrics on $U$.

Clearly, $\varphi:M\to M$ is a bijection, because the extension of $\varphi^{-1}$ to a map on $M$ defines an inverse to $\varphi$. To show that $\varphi$ is an isometry, it therefore suffices to show that $\varphi$ is a distance-preserving map \cite[Theorem I.11.1]{Helgason}. But if $p=\lim p_n$ and $p'=\lim p_n'$ for sequences $p_n,p_n'$ in $U$, then $d_M(\varphi(p),\varphi(p')) = \lim_{n\rightarrow \infty} d_U(\varphi(p_n),\varphi(p_n'))=\lim_{n\rightarrow \infty} d_U(p_n,p_n') = d_M(p,p')$.

By the uniqueness of the extension, it follows that we obtain a well-defined group homomorphism $f:I(U)\rightarrow I(M)$ of Lie groups. As convergence in the compact-open topology is for isometries equivalent to pointwise convergence, cf, e.g., \cite[Section I.4]{KN}, and $U$ is dense in $M$, the homomorphism is continuous and hence smooth, cf. \cite[Theorem 3.39]{Warner}. Thus, our $G$-action $G\to I(U)$ induces, by composition with $f$, a well-defined smooth $G$-action on $M$.

To show that the $G$-action on $M$ is also proper, we need to show that $G$ is closed in the isometry group of $M$ \cite[Theorem 4]{DiazRamos}. For that we note that by properness of the $G$-action on $U$, the $G$-orbits are closed in $U$, and by the slice theorem, also closed in $M$. Thus, if a sequence $g_n$ in $G$ converges as isometries of $M$, then the limit isometry leaves invariant $U$; thus, $g_n$ also converges in $I(U)$, and by properness of the $G$-action on $U$, $g_n$ converges in $G$.
\end{proof}

\begin{rem} An easy example of an action on an open and dense subset which does not extend to the whole manifold is given as follows: Let $M$ be the infinite M\"obius strip $(\RR\times [0,1])/_\sim$, with the boundaries identified via $(t,0)\sim (-t,1)$, and $U=\RR\times (0,1)$. Then the Lie group $\RR$ acts properly on $U$ by translation in the $\RR$-direction, but this action does not extend to $M$.
\end{rem}

\section{Non-existence of singular orbits}\label{sec:singorbits}
In this section, we prove that for suitable classes of split-solvable Lie groups $G$ any proper cohomogeneity-one action of $G$ on a, not necessarily complete, Riemannian manifold $M$ preserving a parallel $\SU(3)$-, $\G_2$- or $\Spin(7)$-structure on $M$, respectively, has only regular orbits. In the $\SU(3)$-case, our result generalizes \cite[Theorem 25]{Conti2}.

We begin by recalling some basic definitions needed in this section.
\begin{defn}
\begin{itemize}
\item
For a Lie algebra $\mfg$, the \emph{ascending central series} $\mfg_k$, $k\in \NN_0$, is recursively defined by
\begin{equation*}
\mfg_{(0)}=\{0\},\quad \mfg_{(k+1)}:=\left\{X\in \mfg\mid [X,\mfg]\subseteq \mfg_{(k)}\right\}
\end{equation*}
for all $k\in \NN_0$. Note that $\mfg_{(1)}$ equals the center $\mathfrak{z}(\mfg)$ of $\mfg$ and that $\mfg$ is nilpotent precisely if there exists some $k\in \NN_0$ with $\mfg_{(k)}=\mfg$.
\item
For an action of a Lie group $ G$ on a manifold $M$ and an element $X$ in the associated Lie algebra $\mfg$, we denote in the following by $\overline{X}$ the induced fundamental vector field on $M$ defined in such a way that $\mfg \to \mathfrak{X}(M)$, $X\mapsto 
\overline{X}$, is an anti-homomorphism of Lie algebras. 
\end{itemize}
\end{defn}
Moreover, we will need the following formula for the value of the exterior derivative of an invariant differential form on fundamental vector fields.
\begin{lemma}\label{lem:dinvariantforms}
Let $M$ be a manifold with an action of a Lie group $ G$ and let $\omega\in \Omega^k M$ be a $ G$-invariant $k$-form on $M$. Then
\begin{equation}\label{eq:differential}
d\omega(\overline{X_0},\ldots,\overline{X_k})=\sum_{0\leq i<j\leq k} (-1)^{i+j}\omega\left(\overline{[X_i,X_j]},\overline{X_0},\ldots,\widehat{\overline{X_i}},\ldots,\widehat{\overline{X_j}},\ldots,\overline{X_k}\right).
\end{equation}
for all $X_0,\ldots, X_k\in \mfg$.
\end{lemma}
\begin{proof}
Since $ G$ preserves $\omega$, we have $\mathcal{L}_{\overline{X_i}}\omega=0$ for all $i\in \{0,\ldots,k\}$. Hence,
\begin{equation*}
\begin{split}
\overline{X_i}(\omega(\overline{X_0},\ldots, \widehat{\overline{X_i}},\ldots,\overline{X_k}))&=\sum_{j=0}^{i-1} (-1)^j \omega\left([\overline{X_i},\overline{X_j}],\overline{X_0},\ldots,\widehat{\overline{X_j}},\ldots,\widehat{\overline{X_i}},\ldots,\overline{X_k}\right)\\
+&\sum_{j=i+1}^k (-1)^{j-1}\omega\left([\overline{X_i},\overline{X_j}],\overline{X_0},\ldots,\widehat{\overline{X_i}},\ldots,\widehat{\overline{X_j}},\ldots,\overline{X_k}\right).
\end{split}
\end{equation*}
Thus, we obtain
\begin{equation*}
\begin{split}
d\omega(\overline{X_0},\ldots,\overline{X_k})=&\sum_{i=0}^k (-1)^i \overline{X_i}(\omega\left(\overline{X_0},\ldots, \widehat{\overline{X_i}},\ldots,\overline{X_k}\right)\\
+&\sum_{0\leq i<j\leq k} (-1)^{i+j} \omega\left([\overline{X_i},\overline{X_j}],\overline{X_0},\ldots, \widehat{\overline{X_i}},\ldots,\widehat{\overline{X_j}},\ldots,\overline{X_k}\right)\\
=&\sum_{0\leq j<i\leq k} (-1)^{i+j}\omega\left([\overline{X_i},\overline{X_j}],\overline{X_0},\ldots,\widehat{\overline{X_j}},\ldots,\widehat{\overline{X_i}},\ldots,\overline{X_k}\right)\\
=&\sum_{0\leq i<j\leq k} (-1)^{i+j}\omega\left(\overline{[X_i,X_j]},\overline{X_0},\ldots,\widehat{\overline{X_i}},\ldots,\widehat{\overline{X_j}},\ldots,\overline{X_k}\right).
\end{split}
\end{equation*}
\end{proof}

The following theorem generalizes \cite[Theorem 25]{Conti2}.

\begin{thm}\label{thm:hyposingorbit} Let $M$ be a six-dimensional manifold with a parallel $\SU(3)$-structure $(\omega,\rho)\in \Omega^2 M\times \Omega^3 M$ preserved by a proper cohomogeneity one action of a five-dimensional Lie group $G$. Then the stabilizer subalgebra $\mfg_p$ of any point $p\in M$
fulfills $\mfg_p\cap \mathfrak{z}(\mfg)=\{0\}$ and if $ G$ is solvable, also $\mfg_p\cap  [\mfg,\mfg]=\{0\}$. In particular, if $ G$ is split-solvable, then all orbits of the action are five-dimensional.
\end{thm}
\begin{proof} We only have to show the first assertion  $\mfg_p\cap \mathfrak{z}(\mfg)=\{0\}$ since the others follow directly from that assertion using Lemma \ref{lem:isocenter}.

To prove this, assume the contrary, i.e. that there exists some point $p\in M$ with $\mfg_p\cap \mathfrak{z}(\mfg)\neq \{0\}$. Let $0\neq X\in \mfg_p\cap \mathfrak{z}(\mfg)$. Then $ G\cdot p$ is a singular orbit and we may fix a normal geodesic $\gamma:(a,b]\to M$ of unit speed with $p=\gamma(b)$ and $\gamma(t)$ being in a regular orbit for all $t\in (a,b)$. Denote by $\left(\alpha_t,\,(\omega_1)_t,\, (\omega_2)_t,\, (\omega_3)_t\right)$ the induced $\SU(2)$-structure on $ G\cdot \gamma(t)$ for $t\in (a,b)$, cf. Proposition \ref{prop:hypoflow}. Since $X$ is contained in the center, the flow equations \eqref{eq:hypoflow} and Lemma \ref{lem:dinvariantforms} give us
\begin{equation*}
\begin{split}
\frac{d}{dt} (\omega_1)_t (\overline{X},\overline{Y})(\gamma(t)) =\, & -d\alpha_t(\overline{X},\overline{Y})(\gamma(t))= \alpha_t\left( \overline{[X,Y]} \right)(\gamma(t))=0.
\end{split}
\end{equation*}
for all $Y\in \mfg$. Thus, $(\omega_1)_t(\overline{X},\overline{Y})$ is constant along the geodesic $\gamma$. Moreover, $\omega= (\omega_1)_t+\alpha_t\wedge dt$ for all $t\in (a,b)$ by Proposition \ref{prop:hypoflow} and so extends also to the singular orbit. Therefore, restricting to $\gamma$, the limit $\lim_{t\to b} \omega(\overline{X},\overline{Y})(\gamma(t))$ exists and is equal to $0$. On the other hand, for all $t\in (a,b)$, $ \omega(\overline{X},\overline{Y})(\gamma(t))= (\omega_1)_t(\overline{X},\overline{Y})(\gamma(t))$. It follows that $\overline{X}_{\gamma(t)}$ is in the kernel of $(\omega_1)_t(\gamma(t))$ for all $t\in (a,b)$.

The kernels of the forms $(\omega_1)_t, (\omega_2)_t$ and $(\omega_3)_t$ at $\gamma(t)$ coincide by Lemma \ref{lem:algprop} (a), so $\overline{X}_{\gamma(t)}$ is also in the kernel of $(\omega_2)_t$ and $(\omega_3)_t$. But then
\begin{equation*}
\begin{split}
d((\omega_3)_t)(\overline{X},\overline Y,\overline Z)(\gamma(t)) =& -(\omega_3)_t(\overline{[X,Y]},\overline{Z})(\gamma(t)) + (\omega_3)_t(\overline{[X,Z]},\overline {Y})(\gamma(t))\\
& - (\omega_3)_t(\overline{[Y,Z]},\overline{X})(\gamma(t))\\ 
=\, &0
\end{split}
\end{equation*}
for all $Y,Z\in \mfg$ since $X$ is in the center and $\overline{X}_{\gamma(t)}$ is in the kernel of $(\omega_3)_t$. Therefore, by the flow equations \eqref{eq:hypoflow}, $(\alpha_t\wedge (\omega_2)_t)(\overline X,\overline Y,\overline Z)$ is constant along $\gamma$.

Now Proposition \ref{prop:hypoflow} gives us $\rho =(\omega_2)_t\wedge \alpha_t-(\omega_3)_t\wedge dt$ on $\bigcup_{t\in (a,b)} G\cdot \gamma(t)$. So the limit $\lim_{t\to b}(\alpha_t\wedge (\omega_2)_t)(\overline X,\overline Y,\overline Z)(\gamma(t))=\lim_{t\to b} \rho(\overline X,\overline Y,\overline Z)(\gamma(t))$ exists and is equal to $0$. Hence we have $(\alpha_t\wedge (\omega_2)_t)(\overline X,\overline Y,\overline Z)(\gamma(t))=0$ for all $Y,Z\in \mfg$ and all $t\in (a,b)$. But this is a contradiction, as $\overline{X}_{\gamma(t)}$ is in the kernel of the $(\omega_2)_t(\gamma(t))$ which has trivial intersection with the kernel of $\alpha_t(\gamma(t))$ by Lemma \ref{lem:algprop} (a).
\end{proof}
For cohomogeneity-one actions on manifolds with parallel $\G_2$-structures, we can show the following theorem.
\begin{thm}\label{thm:hfsingorbit} Let $M$ be a seven-dimensional manifold with a parallel $\G_2$-structure $\varphi\in\Omega^3 M$ preserved by a proper cohomogeneity one action of a six-dimensional Lie group $ G$. If there exists a point $p\in M$ such that $\mfg_p\cap \mathfrak{z}(\mfg)\neq \{0\}$ or if $G$ is solvable and there exists a point $p\in M$ such that 
$\mfg_p\cap [\mfg,\mfg]\neq \{0\}$, then 
$\dim(\mfg_{(2)})=1$. In particular, if $ G$ is split-solvable with $\dim(\mfg_{(2)})\neq 1$ (e.g.\ if $ G$ is nilpotent), then all orbits of the action are six-dimensional.
\end{thm}
\begin{proof}
We only have to show that $\mfg_p\cap \mathfrak{z}(\mfg)\neq \{0\}$ implies $\dim(\mfg_{(2)}))=1$. The other statements follow then by Lemma \ref{lem:isocenter} or since $\mfg_{(k)}=\mfg$ for some $k\in \mathbb{N}$ if $\mfg$ is nilpotent.

For the proof, assume the contrary, i.e.\ that $\dim(\mfg_{(2)})>1$. Let $0\neq X \in \mfg_p\cap \mathfrak{z}(\mfg)$. Then $ G\cdot p$ is a singular orbit and we choose a normal geodesic $\gamma:(a,b]\rightarrow M$ of unit speed with $\gamma(b)=p$ and such that $\gamma(t)$ is in a regular orbit for all $t\in (a,b)$. By Proposition \ref{prop:Hitchinflow}, we get a smooth one-parameter family $(a,b)\ni t\mapsto (\omega_t,\rho_t)$ of half-flat $\SU(3)$-structure on the regular orbits $ G\cdot \gamma(t)$.

Since $\dim(\mfg_{(2)})>1$, we may choose $Y\in \mathfrak{z}(\mfg)=\mfg_{(1)}$ linearly independent of $X$ if $\dim(\mathfrak{z}(\mfg))>1$. Otherwise we choose an arbitrary $Y\in \mfg_{(2)}$ linearly independent of $X$. In both cases, we have $[Y,\mfg]\subseteq \spa{X}$ and so Hitchin's flow equations \eqref{eq:Hitchinflow} give us
\[
\dot \rho_t(\overline X ,\overline Y, \overline Z)(\gamma(t)) = d\omega_t( \overline X ,\overline Y,\overline Z)(\gamma(t))=-\omega_t(\overline{[Y,Z]}, \overline X)=0
\]
for any $Z\in \mfg$. Hence, $\rho_t(\overline X,\overline Y,\overline Z)$ is constant along the normal geodesic $\gamma$. Moreover, by Proposition \ref{prop:Hitchinflow}, we have $\varphi=\omega_t\wedge dt + \rho_t$ on $\bigcup_{t\in (a,b)} G\cdot \gamma(t)$ . Thus, $\lim_{t\rightarrow b}\rho_t(\overline X,\overline Y,\overline Z)(\gamma(t))=\lim_{t\rightarrow b}\varphi(\overline X,\overline Y,\overline Z)(\gamma(t))$ exists and is equal to zero. Hence, $\rho_t(\overline X,\overline Y, \overline Z)(\gamma(t))=0$ for all $t$ and all $Z$ and Lemma \ref{lem:algprop} (b) gives us that $\overline{X}$ and $\overline{Y}$ are complex linearly dependent along $\gamma$ with respect to the complex structure $J_{\rho _t}$ and that $\hat\rho_t(\overline X,\overline Y,\overline Z)(\gamma(t))=0$ for all $t\in (a,b)$ and all $Z\in \mfg$. Thus,
\begin{equation*}
\begin{split}
d\hat\rho_t (\overline X, \overline Y,\overline Z,\overline W)(\gamma(t))=&\,-\hat\rho_t(\overline{[Y,Z]},\overline X,\overline W)+\hat\rho_t(\overline{[Y,W]},\overline X,\overline Z)\\
&\,-\hat\rho_t(\overline{[Z,W]},\overline X,\overline Y)\\
=&\,0
\end{split}
\end{equation*}
for all $Z,\, W\in \mfg$ and all $t\in (a,b)$. Setting $\sigma_t:=\frac{\omega_t^2}{2}$, Hitchin's flow equations \eqref{eq:Hitchinflow} imply that \[\sigma_t(\overline X,\overline Y,\overline Z,\overline W)(\gamma(t))\] is constant along $\gamma$ for fixed $Z,\,W\in \mfg$. As $\star_{\varphi}\varphi=\sigma_t-\widehat{\rho_t}\wedge dt$ on $\bigcup_{t\in (a,b)} G\cdot \gamma(t)$ by Proposition \ref{prop:Hitchinflow}, we conclude, as before, by passing to the singular orbit, that
\[
\sigma_t(\overline X,\overline Y,\overline Z,\overline W)(\gamma(t))=0
\]
for all $Z,\,W\in \mfg$ and all $t\in (a,b)$. But this is impossible: For a fixed $t\in (a,b)$, choose $Z,\,W\in \mfg$ such that $X,\,Y,\,Z,\,W$ are linearly independent but $\overline{Z}_{\gamma(t)},\, \overline{W}_{\gamma(t)}$ are $J_{\rho_t}$-complex linearly dependent. Restricted to $\spa{\overline{X}_{\gamma(t)},\, \overline{Y}_{\gamma(t)},\overline{Z}_{\gamma(t)},\, \overline{W}_{\gamma(t)}}$, $\sigma_t(\gamma(t))=\frac{\omega_t^2}{2}(\gamma(t))$ is some non-zero multiple of the metric volume form as $(g_{(\omega_t,\rho_t)},\linebreak J_{\rho_t},\omega_t)$ is an almost Hermitian structure on $G\cdot\gamma(t)$. Thus, $\sigma_t(\overline X,\overline Y,\overline Z,\overline W)(\gamma(t))\ne 0$, a contradiction.
\end{proof}
We denote by $\mathfrak{n}_{6,5}$ the six-dimensional nilpotent Lie algebra 
with the following non-zero differentials
\begin{equation} \label{n65Eq} de^5=e^{13}+e^{24},\quad de^6=e^{14}-e^{23},\end{equation} 
 where $(e^1,\ldots , e^6)$ is a basis of $\mathfrak{n}_{6,5}^*$. Moreover, we call a semi-direct sum $\mfg=\mfu\rtimes \RR$ \emph{proper} if it is not isomorphic to the Lie algebra direct sum $\mfu\oplus \RR$. Note that then the center of $\mfg$ is contained in $\mfu$ as otherwise $\mfg=\mfu\oplus \RR\cdot X$ as Lie algebras for any element $X$ in the center of $\mfg$ which is not contained in $\mfu$.

\begin{thm}\label{thm:hf2singorbit} Let $M$ be an eight-dimensional manifold with a  parallel $\Spin(7)$-structure $\Phi\in\Omega^4 M$ preserved by a proper cohomogeneity one action of a seven-dimensional split-solvable Lie group $ G$.
\begin{enumerate}
\item[(a)]
If $\mfg=\mathfrak{u}\rtimes \RR$ is a proper semidirect sum with $\mathfrak{u}$ being nilpotent and either $\dim([\mathfrak{u},\mathfrak{u}])\leq 1$ or $\dim([\mathfrak{u},\mathfrak{u}])=2$ and $[\mathfrak{u},\mathfrak{u}]=\mathfrak{z}(\mathfrak{u})$, then all orbits of the action are seven-dimensional.
\item[(b)] If $\mfg=\mathfrak{u}\oplus \RR$ is a direct sum with $\mathfrak{u}\neq \mathfrak{n}_{6,5}$ satisfying the same assumptions as in (a), then all orbits of the action are seven-dimensional.
\item[(c)]
If $G$ is nilpotent and $\mfg_{(k)}\neq 3$ for $k=1,\,2,\, 3$, then all orbits of the action are seven-dimensional.
\end{enumerate}
\end{thm}
\begin{proof} We first prove (a) and (c). 
In both cases, assume that there is a singular orbit and let $p$ a point in this singular orbit. Choose a normal geodesic $\gamma:(a,b]\rightarrow M$ of unit speed such that $\gamma(b)=p$ and $\gamma(t)$ is in a regular orbit for all $t\in (a,b)$. By Proposition \ref{prop:Hitchinflow2}, we have a smooth one-parameter family $(a,b)\ni t\mapsto \varphi_t$ of cocalibrated $\G_2$-structures on the regular orbits $ G\cdot \gamma(t)$. By Lemma \ref{lem:dimsingorbit} and Lemma \ref{lem:isocenter}, $\dim(\mfg_p)=1$ and $\mfg_p$ is central in $\mfg$.
\begin{enumerate}
\item[(a)]
Since the semi-direct sum $\mfg=\mfu\rtimes \RR$ is proper, we have $\mfg_p\subseteq \mfu$. Set $V:=[\mfu,\mfu]+\mfg_p$. Then $1\leq \dim(V)\leq 2$ and $V\subseteq \mathfrak{z}(\mfu)$. Choose $0\neq X \in \mfg_p$. If $\dim(V)=1$, then choose an arbitrary $Y\in \mfu$ linearly independent of $X$. Otherwise, choose $Y\in V$ linearly independent of $X$. In both cases, $[X,\mfu]=\{0\}$, $[Y,\mfu]\subseteq \spa{X}$ and $[\mfu,\mfu]\subseteq \spa{X,Y}$. Consequently, using the flow
equation \eqref{eq:Hitchinflow2} and Lemma \ref{lem:dinvariantforms},
\begin{equation*}
\begin{split}
\left(\star_{\varphi_t} \varphi_t\right)'(\overline X,\overline Y,\overline Z, \overline W)(\gamma(t))=\,& d\varphi_t(\overline X,\overline Y,\overline Z, \overline W)(\gamma(t))=0
\end{split}
\end{equation*}
for all $Z,\, W\in  \mfu$. This shows that $\star_{\varphi_t} \varphi_t(\overline X,\overline Y,\overline Z, \overline W)$ is constant along $\gamma(t)$. Going to the boundary, using that
\begin{equation*}
\star_{\varphi_t} \varphi_t(\overline X,\overline Y,\overline Z, \overline W)(\gamma(t))=\Phi(\overline X,\overline Y,\overline Z, \overline W)(\gamma(t)))
\end{equation*}
by Proposition \ref{prop:Hitchinflow2} for all $t\in (a,b)$, we see that the constant has to be zero. Fix now $t\in (a,b)$. Then there exists a basis $v_1,\ldots,v_7$ of $T_{\gamma(t)} G\gamma(t)$ such that in the dual basis
\begin{equation*}
\star_{\varphi_t} \varphi_t(\gamma(t))=v^{1234}+v^{1256}+v^{3456}+v^{1367}+v^{1457}+v^{2357}-v^{2467}.
\end{equation*}
Since $\G_2$ acts transitively on the set of all $2$-planes in $\RR^7$, we may assume that $\spa{\overline{X}_{\gamma(t)},\overline{Y}_{\gamma(t)}}=\spa{v_1,v_2}$ and so that $\star_{\varphi_t} \varphi_t(\overline X,\overline Y,\cdot, \cdot)(\gamma(t))$ is a non-zero multiple of $\omega:=v^{34}+v^{56}$. $\omega$ has rank four on $T_{\gamma(t)} G\gamma(t)$ and so is non-zero when restricted to any six-dimensional subspace $U\subseteq T_{\gamma(t)} G\gamma(t)$ since otherwise $\omega(U,\cdot)$ is contained in the annihilator $\Ann{U}$ of $U$ which would imply that the rank of $\omega$ was at most two. Hence, the restriction of $\star_{\varphi_t} \varphi_t(\overline X,\overline Y,\cdot, \cdot)(\gamma(t))$ to $\left\{\overline{X}_{\gamma(t)}|X\in \mfu\right\} \subseteq T_{\gamma(t)} G\gamma(t)$ is non-zero, a contradiction.
\item[(c)]
Let $k\in \{1,2,3\}$ be the smallest number such that $\dim(\mfg_{(k)})>3$. Choose $0\neq X\in \mfg_p\subseteq \mathfrak{z}(\mfg)=\mfg_{(1)}$ and $0\neq Y\in \mfg_{(2)}$ linearly independent of $X$ such that $Y\in \mfg_{(1)}$ if $\dim(\mfg_{(1)})\geq 2$. Let $Z\in \mfg_{(k)}$ be linearly independent of $X$ and $Y$. Then $[Y,W]\subseteq \spa{X}$ and $[Z,W]\subseteq \spa{X,Y}$ for all $W\in \mfg$. Thus,
\begin{equation*}
\begin{split}
\left(\star_{\varphi_t} \varphi_t\right)'(\overline X,\overline Y,\overline Z, \overline W)(\gamma(t))=\,& d\varphi_t(\overline X,\overline Y,\overline Z, \overline W)(\gamma(t))=0
\end{split}
\end{equation*}
for all $W\in \mfg$. As in (a), we obtain $\left(\star_{\varphi_t} \varphi_t\right)(\overline X,\overline Y,\overline Z, \overline W)(\gamma(t))=0$ for all $W\in \mfg$ and all $t\in (a,b)$. Hence, $\dim(\mfg_{(k)})=3$ by Lemma \ref{lem:algprop} (c), a contradiction.
\end{enumerate}
Now we prove (b). Going through the list of all nilpotent Lie algebras up to dimension six, cf.\ \cite{Mag} or \cite{PSWZ}, we see that (a) applies exactly to the following six six-dimensional nilpotent Lie algebras $\mfu$: $\RR^6$, $\mfh_3\oplus \RR^3$, $\mfh_3\oplus \mfh_3$, $A_{5,4}\oplus \RR$, $\mathfrak{n}_{6,4}$, $\mathfrak{n}_{6,5}$. Here, $\RR^k$ is the $k$-dimensional Abelian  Lie algebra, $\mfh_3$ is the three-dimensional Heisenberg Lie algebra, $A_{5,4}$ is the Lie algebra with the same name in \cite{PSWZ} (in \cite{Mag}, the Lie algebra is called $\mfg_{5,1}$) and with the only non-zero differential $de^1=e^{24}+e^{35}$. Moreover, $\mathfrak{n}_{6,4}$ and $\mathfrak{n}_{6,5}$ are the fourth and fifth indecomposable nilpotent Lie algebra of dimension six in the list given in \cite{Mag} (in \cite{PSWZ}, they are named $A_{6,4}$ and $A_{6,5}$), where we note that the parameter $\gamma$ in $\mathfrak{n}_{6,5}$ can be chosen to be equal to $-1$, which corresponds to the differentials (\ref{n65Eq}). $\mathfrak{n}_{6,4}$ is given by the following non-zero differentials $de^5=e^{12}, de^6=e^{13}+e^{24}$. The cases $\mfu\in \left\{\RR^7,\mfh_3\oplus \RR^4\right\}$ are covered by (c). If $\mfu\in \left\{\mfh_3\oplus \mfh_3, A_{5,4}\oplus \RR, \mathfrak{n}_{6,4}\right\}$, then $\mfu \oplus \RR$ is isomorphic to a proper semi-direct sum of the form $(\mfh_3\oplus \RR^3)\rtimes \RR$ and so these cases are covered by (a).
\end{proof}
Theorem \ref{thm:hf2singorbit} and Lemma \ref{lem:6to7to8} imply
\begin{cor}\label{cor:8to7}
Let $M$ be a seven-dimensional manifold with a parallel $\G_2$-structure preserved by a proper cohomogeneity-one action of a six-dimensional split-solvable Lie group $ G$. If $\mfg=\mfu\rtimes \RR$ for a five-dimensional nilpotent Lie algebra $\mfu$ with $\dim([\mfu,\mfu])\leq 1$, then all orbits of the action are six-dimensional.
\end{cor}
\section{Flatness}\label{sec:flatness}
In this section, we prove that a complete Riemannian manifold with a parallel $\SU(3)$-, $\G_2$- or $\Spin(7)$-structure which is preserved by a proper cohomogeneity one action of a split-solvable Lie group with properties as in the last section is flat. Here, as throughout this section, we assume without explicitly mentioning that the Riemannian metric is the one induced by the parallel $\SU(3)$-, $\G_2$- or $\Spin(7)$-structure.

The 
main result of this section, Theorem \ref{thm:flatness} below, will be an immediate consequence of the following proposition and the results of the last section.
\begin{prop}\label{prop:flatness}
Let $H\in \{\SU(3),\G_2,\Spin(7)\}$ and let $M$ be a complete Riemannian manifold of appropriate dimension with a parallel $H$-structure preserved by a proper cohomogeneity one action of a $(\dim(M)-1)$-dimensional Lie group $G$. If all orbits are of 
codimension one, then $M$ is flat.
\end{prop}
\begin{proof}
Thanks to Lemma \ref{lem:6to7to8}, we only have to do the proof for $H=\Spin(7)$. For that we combine the arguments of \cite{Stock} and \cite{Conti2}. Recall that for any $\G_2$-structure $\varphi\in \Omega^3 N$ on a seven-dimensional manifold $N$ there exists $\mathcal{T}\in \mathrm{End}(TM)$, called the \emph{intrinsic torsion of $\varphi$}, such that $\nabla^g_X \varphi=-\mathcal{T}(X)\hook \star_\varphi \varphi$ for any vector field $X\in \mathfrak{X}(N)$, cf. \cite[p.\ 542]{Bryant}.

Let $M$ now be a complete eight-dimensional Riemannian manifold with a parallel $\Spin(7)$-structure $\Phi$ preserved by a proper cohomogeneity one action of a seven-dimensional Lie group $ G$ such that all orbits are seven-dimensional. Consider a normal geodesic $\gamma:\RR\rightarrow M$ of unit speed. Then
\begin{equation*}
f:G\times \RR\rightarrow M,\quad f(g,t):=g\cdot \gamma(t)
\end{equation*}
is a $ G$-equivariant local diffeomorphism. Pulling $\Phi$ back to $ G\times \RR$ via $f$, we get a parallel $\Spin(7)$-structure $\tilde\Phi$ on $ G\times \RR$ which is preserved by the natural left action of $ G$ on $ G\times I$. Moreover, $f$ is a local isometry for the induced Riemannian metric $g$ on $ G\times \RR$ and $g$ is of the form $g=g_t\oplus dt^2$. Hence, it suffices to prove the statement for $( G\times \RR,g)$ and $\tilde\Phi$. Let $\varphi_t$ be the left-invariant cocalibrated $\G_2$-structure induced on $ G\times \{t\}\cong G$ by Proposition \ref{prop:Hitchinflow2}. Then $g_t$ is the Riemannian metric induced by $\varphi_t$ on $ G\times \{t\}\cong G$. If $\mathcal{T}_t$ denotes the intrinsic torsion of $\varphi_t$, we have 
\begin{equation}\label{eq:metricevolution}
\dot{g}_t(X,Y)=2 g_t(\mathcal{T}_t(X),Y)
\end{equation}
for $X,Y\in \mathfrak{X}(M)$ by \cite[Theorem 3.2]{Stock}. As $\tilde\Phi$ is parallel, the holonomy of $g$ is contained in $\Spin(7)$ and so $( G\times \RR,g)$ is Ricci-flat. Since the geodesic $ \RR \ni t\mapsto (e,t)\in M$ is a line, the Cheeger-Gromoll Splitting Theorem shows that $( G\times \RR,g)$ is the Riemannian product of $( G,g_0)$ and $(\RR,dt^2)$. In particular, $g_t=g_0$ is constant and so Equation (\ref{eq:metricevolution}) gives $\mathcal{T}_0=0$. Thus, $\varphi_0$ is parallel and so the holonomy of $( G,g_0)$ is contained in $\G_2$. In particular, $(\G,g_0)$ is a Ricci-flat homogeneous space and so flat by \cite{AK}. But then also $( G\times \RR,g)$ is flat.
\end{proof}
Proposition \ref{prop:flatness} and the results of Section \ref{sec:singorbits} imply the following theorem, where we note that part (a) generalizes \cite[Corollary 26]{Conti2}.
\begin{thm}\label{thm:flatness}
\begin{enumerate}
\item[(a)]
Let $M$ be a complete six-dimensional Riemannian manifold with parallel $\SU(3)$-structure preserved by a proper cohomogeneity one action of a five-dimensional split-solvable Lie group $ G$. Then $M$ is flat.
\item[(b)]
Let $M$ be a complete seven-dimensional Riemannian manifold with  parallel $\G_2$-structure preserved by a proper cohomogeneity one action of a six-dimensional split-solvable Lie group $ G$. If $\dim\mfg_{(2)}\neq 1$ (e.g.\ if $G$ is nilpotent) or $\mfg=\mfu\rtimes \RR$ with a five-dimensional nilpotent Lie algebra $\mfu$ satisfying $\dim([\mfu,\mfu])\leq 1$, then $M$ is flat.
\item[(c)]
Let $M$ be a complete eight-dimensional Riemannian manifold with  parallel $\Spin(7)$-structure preserved by a proper cohomogeneity one action of a seven-dimensional split-solvable Lie group $ G$. If the conditions in Theorem \ref{thm:hf2singorbit} (a), (b) or (c) are satisfied, then $M$ is flat.
\end{enumerate}
\end{thm} 

\section{Maximally symmetric solutions of the Hitchin flow on $\SL(2,\CC)$}\label{sec:example}
In this section we consider the Hitchin flow for maximally symmetric initial left-invariant half-flat $\SU(3)$-structures $(\omega,\rho)\in \Omega^2 G\times \Omega^3 G$ on the Lie group $ G=\SL(2,\CC)$. We compute for which initial values the Hitchin flow gives solutions on a finite interval $(a,b)$ which can be extended at one of the two boundary points.  The fact that such a phenomenon occurs is in contrast to the results for certain types of split-solvable Lie groups in Section \ref{sec:singorbits}. 
 The incomplete Riemannian metrics obtained in this way all turn out to be homothetic to each other, as well as to a metric described previously by Bryant and Salamon \cite[Theorem on p.\ 840]{BryantSalamon}. The purpose of this section is on the one hand to give a new description of this metric from the point of view of the Hitchin flow, and on the other hand to put it into a more general context by solving the Hitchin flow for all maximally symmetric initial values.

As usual, we identify all left-invariant tensor fields with the corresponding tensors on the associated Lie algebra $\mfg=\mathfrak{sl}(2,\CC)$. Since $\mathfrak{sl}(2,\CC)\cong \mathfrak{so}(3,1)$, we may choose a basis $(e_1,\ldots,e_6)$ of $\mfg$ such that the differentials $\left(de^1,de^2,de^3,de^4,de^5,de^6\right)$ of the dual basis $\left(e^1,\ldots,e^6\right)$ are given by
\begin{equation*}
\left(e^{23}-e^{56},-e^{13}+e^{46},e^{12}-e^{45},e^{26}-e^{35},-e^{16}+e^{34},e^{15}-e^{24}\right).
\end{equation*}
As $\SU(2)$ is a maximal compact subgroup of $\SL(2,\CC)$, we get maximal symmetry if we assume that the initial half-flat $\SU(3)$-structure on $\mathfrak{sl}(2,\CC)$ is Ad-invariant under $\SU(2)$. We first determine all such $\SU(3)$-structures. For that, observe that $\mfg=\mfg^*=V\oplus V$ as $\SU(2)$-modules with $V$ being the $3$-dimensional adjoint representation of $\SU(2)$ which equals via the universal covering $\SU(2)\rightarrow \mathrm{SO}(3)$ the standard $\mathrm{SO}(3)$-representation on $\RR^3$. Hence,
\begin{equation*}
\Lambda^2\mfg^*=\Lambda^2 (V\oplus V)=2\Lambda^2 V\oplus V\otimes V=3\Lambda^2 V\oplus S^2 V=3 V\oplus S^2 V
\end{equation*}
and
\begin{equation*}
\Lambda^3\mfg^*=\Lambda^3 (V\oplus V)=2\Lambda^3 V\oplus 2 \Lambda^2 V\otimes V=2\RR\oplus 2 V\otimes V=2\RR\oplus 2 V\oplus 2 S^2 V
\end{equation*}
as $\SU(2)$-modules. As $V$ has no non-zero $\SU(2)$-invariant elements and the $\SU(2)$-invariant elements of $S^2 V$ are multiples of the metric $g_0$, we have
\begin{align*}
\left(\Lambda^2\mfg^*\right)^{\SU(2)} &=  \RR\cdot g_0=\spa{e^{14}+e^{25}+e^{36}},\\
\left(\Lambda^3\mfg^*\right)^{\SU(2)}&=  2\RR\oplus 2 \RR\cdot g_0=\spa{e^{123},e^{456},e^{126}-e^{135}+e^{234}, e^{156}-e^{246}+e^{345}}.
\end{align*}
For arbitrary $a,\,b_1,\,b_2,\,b_3,\,b_4\in \RR$, the forms
\begin{equation*}
\begin{split}
\omega&:=  a \left(e^{14}+e^{25}+e^{36}\right),\\
\rho&:=  b_1 e^{123}+b_2 e^{456}+b_3\left(e^{126}-e^{135}+e^{234}\right)+b_4 \left(e^{156}-e^{246}+e^{345}\right)
\end{split}
\end{equation*}
fulfill $\rho\wedge \omega=0$ and $d(\omega^2)=0$. Moreover, $\rho$ is closed precisely when $b_4=-b_1$. We set $b_4:=-b_1$ in the following.

Next, we want to compute $\lambda(\rho)$ and
  $J_{\rho}$. Therefore, we first have to compute $K_{\rho}\in
  \mathrm{End}(\mfg)\otimes \Lambda^6 \mfg^*$ via Equation
  \eqref{eq:Krho}. As $\rho$ is $\SU(2)$-invariant, $K_{\rho}$ is
  $\SU(2)$-invariant as well. Consequently, we only have to compute
  $K_{\rho}(e_1)$ and $K_{\rho}(e_4)$ to determine $K_{\rho}$. Since
   $K_{\rho}(e_1)=\left(b_1(b_2+b_3)e_1+2(b_1^2+b_3^2)e_4\right)\otimes e^{123456}$ and $K_{\rho}(e_4)=\left(2(b_2 b_3-b_1^2) e_1-b_1(b_2+b_3)e_4\right)\otimes e^{123456}$, we obtain
\begin{equation*}
\begin{split}
K_{\rho}=&\left(\sum_{i=1}^3 e^i\otimes \left(b_1(b_2+b_3)e_i+2(b_1^2+b_3^2)e_{i+3}\right)\right)\otimes e^{123456}\\
        + &\left(\sum_{i=1}^3 e^{i+3}\otimes \left(2(b_2 b_3-b_1^2)e_i-b_1(b_2+b_3)e_{i+3}\right)\right)\otimes e^{123456}.
\end{split}
\end{equation*}
Hence, by Equation \eqref{eq:lambda},
\begin{equation*}
\begin{split}
\lambda(\rho)=& \frac{1}{6}\mathrm{tr}((K_{\rho}\otimes \id_{\Lambda^6 V^*})\circ K_{\rho} )\\
= &\left(b_1^2(b_2+b_3)^2-4(b_1^2+b_3^2)(b_1^2-b_2b_3)\right)\left(e^{123456}\right)^{\otimes 2}.
\end{split}
\end{equation*}
We set
\begin{equation}\label{eq:lambdaexplizit}
\lambda(b_1,b_2,b_3):=b_1^2(b_2+b_3)^2-4(b_1^2+b_3^2)(b_1^2-b_2b_3).
\end{equation}
To get an $\SU(3)$-structure compatible with the orientation
  given by $\omega^3$, we need $\lambda(b_1,b_2,b_3)<0$ by Lemma \ref{lem:characterizationSU3} and then the normalization condition $\sqrt{-\lambda(\rho)}=\frac{1}{3}\omega^3$ reads
\begin{equation*}
|a|= 2^{-\frac{1}{3}}\left(-\lambda(b_1,b_2,b_3)\right)^{\frac{1}{6}}.
\end{equation*}
Using Equation \eqref{eq:Jrho}, the induced almost complex structure $J_{\rho}$ is given by
\begin{equation*}
\begin{split}
J_{\rho}=&-\frac{\mathrm{sgn}(a)}{\sqrt{-\lambda(b_1,b_2,b_3)}}\sum_{i=1}^3 e^i\otimes \left(b_1(b_2+b_3)e_i+2(b_1^2+b_3^2)e_{i+3}\right)\\
&-\frac{\mathrm{sgn}(a)}{\sqrt{-\lambda(b_1,b_2,b_3)}}\sum_{i=1}^3 e^{i+3}\otimes \left(2(b_2 b_3-b_1^2)e_i-b_1(b_2+b_3)e_{i+3}\right).
\end{split}
\end{equation*}
Hence, the induced metric $g_{(\omega,\rho)}=\omega(J_{\rho}\cdot,\cdot)$ equals
\begin{equation}\label{eq:metric}
\begin{split}
g_{(\omega,\rho)}=\frac{2^{\frac{2}{3}}}{(-\lambda(b_1,b_2,b_3))^{\frac{1}{3}}}
\sum_{i=1}^3&\left((b_1^2+b_3^2)e^{i}\otimes e^{i})+
(b_1^2-b_2b_3)e^{i+3}\otimes e^{i+3}\right.\\ 
& \left. -\frac{b_1(b_2+b_3)}{2}(e^i\otimes e^{i+3}+e^{i+3}\otimes
e^i)\right).
\end{split}
\end{equation}
We must have $b_1^2+b_3^2>0$ as otherwise $\lambda(b_1,b_2,b_3)\geq 0$. Moreover,
\begin{equation*}
(b_1^2+b_3^2) (b_1^2-b_2 b_3)-\frac{b_1^2(b_2+b_3)^2}{4}=-\frac{\lambda(b_1,b_2,b_3)}{4}>0
\end{equation*}
and so $g_{(\omega,\rho)}$ is always positive definite. Thus, by Lemma \ref{lem:characterizationSU3}, we have obtained
\begin{lemma}\label{lem:allinvarianthf}
In the basis $(e_1,\ldots,e_6)$ of $\mathfrak{sl}(2,\CC)$ given above, the set of all $\SU(2)$-invariant half-flat $\SU(3)$-structures $(\omega,\rho)$ is given by
\begin{equation*}
\begin{split}
\omega=\ &\omega_{\epsilon, b_1,b_2,b_3}:= 2^{-\frac{1}{3}}\epsilon\cdot\left(-\lambda(b_1,b_2,b_3)\right)^{\frac{1}{6}} \left(e^{14}+e^{25}+e^{36}\right),\\
\rho=\ & \rho_{b_1,b_2,b_3}:=b_1 e^{123}+b_2 e^{456}+b_3\left(e^{126}-e^{135}+e^{234}\right)- b_1 \left(e^{156}-e^{246}+e^{345}\right)
\end{split}
\end{equation*}
for arbitrary $b_1,b_2,b_3\in \RR$ with $\lambda(b_1,b_2,b_3)<0$ and $\epsilon\in \{-1,1\}$, where $\lambda(b_1,b_2,b_3)$ is defined by Equation (\ref{eq:lambdaexplizit}).
\end{lemma}
 Note that a change of sign in $\epsilon$ does not change the
  metric $g_{(\omega,\rho)}$.

Next, we want to solve the Hitchin flow for the initial values
$\omega_0=\omega_{\epsilon,b_1,b_2,b_3},\,
\rho_0=\rho_{b_1,b_2,b_3}$. For this we note that during the flow
$\rho_t$ stays closed and $(\omega_t,\rho_t)$ stay $\SU(2)$-invariant
as $\SU(2)$ acts on $\mathfrak{sl}(2,\CC)$ by automorphisms,
cf. Remark \ref{rem:invariant}. Hence, the maximal solution
$(\omega_t,\rho_t)$ of the Hitchin flow with the mentioned initial
values has to be of the form 
\begin{equation*}
\begin{split}
\rho_t=& y_1(t) e^{123}+y_2(t) e^{456}+y_3(t) \left(e^{126}-e^{135}+e^{234}\right)- y_1(t) \left(e^{156}-e^{246}+e^{345}\right)\\
\omega_t= &  2^{-\frac{1}{3}}\epsilon\left(-\lambda(y_1(t),y_2(t),y_3(t))\right)^{\frac{1}{6}} \left(e^{14}+e^{25}+e^{36}\right)
\end{split}
\end{equation*}
for smooth functions $y_1,\, y_2,\, y_3: (a,b)\rightarrow \RR$, $-\infty \leq a<0<b\leq \infty$ fulfilling $y_i(0)=b_i$ for $i=1,\,2,\,3$, where $(a,b)$ is the maximal interval of existence. The flow equation $\dot{\rho_t}=d\omega_t$ reads
\begin{equation*}\label{eq:floweqs1}
\dot{y_1}=0,\, \dot{y_2}=-3\cdot 2^{-\frac{1}{3}} \epsilon\cdot(-\lambda(y_1,y_2,y_3))^{\frac{1}{6}},\, \dot{y_3}=- 2^{-\frac{1}{3}} \epsilon\cdot(-\lambda(y_1,y_2,y_3))^{\frac{1}{6}}
\end{equation*}
and so we obtain that $y_1\equiv b_1$, $y_2=3 y_3+b_2-3b_3$ and $y_3:(a,b)\rightarrow \RR$ is a maximal solution of the initial value problem
\begin{equation}\label{eq:flow}
\dot{x}=-  2^{-\frac{1}{3}}\epsilon\cdot\left(-\lambda(b_1,3x+b_2-3b_3,x)\right)^{\frac{1}{6}},\quad x(0)=b_3.
\end{equation}
Note that then the second flow equation $\frac{d}{dt}\left(\frac{\omega_t^2}{2}\right)=d\widehat{\rho_t}$ is automatically fulfilled. To avoid indices, we set $x:=y_3$ in the following.

Next, we prove that the maximal interval of existence $(a,b)$ is
finite. By separation of variables in the equation
$$ \frac{dy}{ds}=-2^{-\frac13}\epsilon f(y)^{-\frac16},$$
where $f(x):=-\lambda(b_1,3x+b_2-3b_3,x)$, we get
\begin{equation*}
|t|=\left|\int_{0}^t ds\right|=\left|-2^{\frac{1}{3}}\epsilon \int_{b_3}^{x(t)} f(y)^{-\frac{1}{6}} dy\right|= 2^{\frac{1}{3}} \left|\int_{b_3}^{x(t)}  f(y)^{-\frac{1}{6}} dy \right|
\end{equation*}
From Equation \eqref{eq:lambdaexplizit}, $f$ is a polynomial of degree four with leading coefficient equal to $-12$. Hence, $f(x)$ is negative for all $x$ whose absolute value is sufficiently large and as $f(b_3)>0$, there exist $x_1<b_3<x_2$ with $f(x_1)=f(x_2)=0$ and $f(x)>0$ for all $x\in (x_1,x_2)$. By the standard theory of ordinary differential equations, $x(t)$ tends to $x_1$ or $x_2$ when $t$ tends to $a$ and it tends to the other $x_i$ when $t$ tends to $b$. Hence, it suffices to show that $\lim_{x\rightarrow x_1^+}\int_{b_3}^{x}  f(y)^{-\frac{1}{6}} dy$ and $\lim_{x\rightarrow x_2^-}\int_{b_3}^{x}  f(y)^{-\frac{1}{6}} dy$ are finite. But $x_1$ and $x_2$ are zeros of $f$ of multiplicity at most four and so there exist constants $C_1, C_2>0$ such that $f(x)\geq C_1 (x-x_1)^4$ for all $x$ near $x_1$ and $f(x)\geq C_2 (x-x_2)^4$ for all $x$ near $x_2$. Thus, the integrand in the above integrals is less or equal to $C_i^{-\frac{1}{6}} (x-x_i)^{-\frac{2}{3}}$ near $x_i$ and, consequently,
$\lim_{x\rightarrow x_1^+}\int_{b_3}^{a}  f(y)^{-\frac{1}{6}} dy$ and $\lim_{x\rightarrow x_2^-}\int_{b_3}^{x} f(y)^{-\frac{1}{6}} dy$ are finite. Hence, $a$ and $b$ are both finite. Summarizing these results, we have
\begin{lemma}\label{lem:hfsolutions}
The maximal solution $(a,b)\ni t\mapsto (\omega_t,\rho_t)$ of the Hitchin flow on $\mathfrak{sl}(2,\CC)$ with initial value at $t=0$ equal to $(\omega_{\epsilon,b_1,b_2,b_3},\rho_{b_1,b_2,b_3})$ as in Lemma \ref{lem:allinvarianthf} is defined on a finite interval $(a,b)$ and is explicitly given by
\begin{equation*}
\begin{split}
\omega_t&=   2^{-\frac{1}{3}}\epsilon \cdot \left(-\lambda(b_1,3x(t)+b_2-3b_3,x(t))\right)^{\frac{1}{6}} \left(e^{14}+e^{25}+e^{36}\right)\\
\rho_t&= b_1 e^{123}+(3x(t)+b_2-3b_3) e^{456}+x(t) \left(e^{126}-e^{135}+e^{234}\right)\\
&\qquad  - b_1 \left(e^{156}-e^{246}+e^{345}\right),\\
\end{split}
\end{equation*}
where $x:(a,b)\rightarrow \RR$ is the maximal solution of the initial value problem
\begin{equation*}
\dot{x}=-  2^{-\frac{1}{3}}\epsilon\cdot\left(-\lambda(b_1,3x+b_2-3b_3,x)\right)^{\frac{1}{6}},\ \ \, x(0)=b_3.
\end{equation*}
\end{lemma}
Now we can state and prove the main result of this section:
\begin{thm}\label{thm:hfextended}
Let $(\SL(2,\CC)\times (a,b),g_t+dt^2)$, $-\infty<a<0<b<\infty$, be the Riemannian manifold obtained from the maximal solution of the Hitchin flow on $\SL(2,\CC)$ with initial $\SU(2)$-invariant left-invariant half-flat $\SU(3)$-structure \linebreak $(\omega_{\epsilon,b_1,b_2,b_3},\rho_{b_1,b_2,b_3})$ at $t=0$ as in Lemma \ref{lem:allinvarianthf}.

Then $(\SL(2,\CC)\times (a,b),g_t+dt^2)$ cannot be realized as the regular part of a Riemannian manifold with a proper isometric $\SL(2,\CC)$-action of cohomogeneity one with two singular orbits. However, it can be realized as the regular part of a Riemannian manifold $(M,g)$ with a proper isometric $\SL(2,\CC)$-action of cohomogeneity one with precisely one singular orbit if and only if $b_1=0$. All these Riemannian manifolds $(M,g)$ are non-complete and homothetic to the Bryant-Salamon metric with holonomy $\G_2$ on the spin bundle over hyperbolic $3$-space described in \cite{BryantSalamon}.
\end{thm}
\begin{proof}
We first consider the case $b_1\neq 0$. We argue by contradiction and assume that $(\SL(2,\CC)\times (a,b),g_t+dt^2)$ can be realized as the regular part of a Riemannian manifold $(M,g)$ with a proper isometric $\SL(2,\CC)$-action of cohomogeneity one and at least one singular orbit. Then $\lim_{t\rightarrow b^-} (e,t)$ or $\lim_{t\rightarrow a^+} (e,t)$ has to exist in $M$ and lie in a singular orbit and so $\lim_{t\rightarrow b^-} g(\overline{X},\overline{X}) (e,t)$ or $\lim_{t\rightarrow a^+} g(\overline{X},\overline{X}) (e,t)$ has to exist and has to be finite for all $X\in \mathfrak{sl}(2,\CC)$. But Equation (\ref{eq:metric}) and Lemma \ref{lem:hfsolutions} give
\begin{equation*}
g_t(e_1,e_1)=g(\overline{e_1},\overline{e_1})(e,t)=2^{\frac{2}{3}}\frac{ b_1^2+x(t)^2}{(-\lambda(b_1,3x(t)+b_2-3b_3,x(t)))^{\frac{1}{3}}}.
\end{equation*}
for all $t\in (a,b)$ and, as argued above, $\lambda(b_1,3x(t)+b_2-3b_3,x(t))$ tends to zero for $t$ tending to $a$ and $b$. Since the numerator is always non-zero, $g_t(e_1,e_1)$ tends to infinity at both boundary points of $(a,b)$, a contradiction.

Assume from now on that $b_1=0$. Since
$\lambda(0,3x+b_2-3b_3,x)=4x^3(3x+b_2-3b_3)$ has exactly two
  distinct zeroes: $0$ and $b_3-\frac{1}{3} b_2$ (which is non-zero
  since otherwise $\lambda(0,b_2,b_3)=4b_3^3 b_2=12 b_3^4\geq 0)$, we
  know from Lemma \ref{lem:hfsolutions} that 
$$\lim_{t\rightarrow a} x(t)=z_1\mbox{ and }\lim_{t\rightarrow b}
x(t)=z_2,$$
where $\{z_1,z_2\}=\{0,b_3-\frac{1}{3} b_2\}$, Equation (\ref{eq:metric}) yields
\begin{equation*}
g_t=\sum_{i=1}^3 \frac{x(t)}{(3b_3-b_2-3x(t))^{\frac{1}{3}}} e^i\otimes e^i+\sum_{j=4}^6 (3b_3-b_2-3x(t))^{\frac{2}{3}} e^j\otimes e^j.
\end{equation*}
Consequently, $g_t(e_1,e_1)$ tends to $\infty$ at one of the boundary points of $(a,b)$ and so $\SL(2,\CC)\times (a,b)$ cannot be realized as the regular part of a Riemannian manifold with a proper cohomogeneity-one action of $\SL(2,\CC)$ with two singular orbits.

As $\dot{x}=- \epsilon \left(x^3(3b_3-b_2-3x)\right)^{\frac{1}{6}}$ by
Equation (\ref{eq:flow}), $(\SL(2,\CC)\times (a,b),g_t+dt^2)$ is the
expression of the metric $ h_{b_2,b_3}$ through the change of variable
$t\mapsto x(t)$ in $\left(\SL(2,\CC)\times J_{b_2,b_3}, h_{b_2,b_3}\right)$ with
\begin{equation*}
h_{b_2,b_3}:=\sum_{i=1}^3 \frac{x}{(3b_3-b_2-3x)^{\frac{1}{3}}} e^i\otimes e^i+\sum_{j=4}^6 (3b_3-b_2-3x)^{\frac{2}{3}} e^j\otimes e^j+\frac{dx^2}{x (3b_3-b_2-3x)^{\frac{1}{3}}}
\end{equation*}
and $J_{b_2,b_3}:=\left(0,b_3-\frac{b_2}{3}\right)$ or $J_{b_2,b_3}:=\left(b_3-\frac{b_2}{3},0\right)$ depending on whether $b_3-\frac{b_2}{3}>0$ or not. $\left(\SL(2,\CC)\times J_{b_2,b_3},h_{b_2,b_3}\right)$ is homothetic to $\left(\SL(2,\CC)\times J_{-1,1}, h_{-1,1}\right)$ via the homothety $(A,x)\mapsto \left(A, \frac{4x}{3b_3-b_2}\right)$.

Hence, it suffices to consider $(\epsilon,b_2,b_3)=(1,-1,1)$ and to show that for these parameter values, $(\SL(2,\CC),g_t+dt^2)$ can be realized as the regular part of a, necessarily incomplete, Riemannian manifold $(M,g)$ with a proper cohomogeneity-one $\SL(2,\CC)$-action with exactly one singular orbit. In this case,
\begin{equation} \label{eq:xdot}
\dot{x}=-\sqrt{x}\left(4-3x\right)^{\tfrac{1}{6}},\quad x(0)=1.
\end{equation}
In particular, $x$ is strictly decreasing and $x$ tends to $0$ when
reaching the boundary point $b$. Moreover, we have $\lim_{t\to
  b^-}g_t(e_i,e_i)=0$ for $i=1,2,3$. The elements $e_1,e_2,e_3$ span
the subalgebra $\su(2) \subset \msl(2,\CC)$. 

Because the action of
$\SL(2,\CC)$ on $\SL(2,\CC)\times I=\SL(2,\CC)\times (a,b)$ is free,
this implies that the candidate for the partial completion at $b$ is
the manifold $M=\SL(2,\CC)\times_{\SU(2)} U$ with $U:=B_{\epsilon}(0)$
being the open ball of radius $\epsilon:=b-a$ in $V=\RR^4$ with
respect to the canonical inner product on $\RR^4$ and the
representation of $\SU(2)\cong S^3$ on $V=\RR^4\cong \HH$ being the
standard one.  \begin{rem} $M$ can be seen as the space of
    vectors of length  less than $b-a$ in the spin bundle over the
    hyperbolic 3-space $\SL(2,\CC)/\SU(2)$.\end{rem}
We may assume that $\SL(2,\CC)\times (a,b)$ is embedded into
$M$ via  $(A,t)\mapsto [(A,(b-t)v_1)]\in
  \SL(2,\CC)\times_{\SU(2)} U\subseteq \SL(2,\CC)\times_{\SU(2)} V$,
where $(v_1,\ldots,v_4)$ is the standard basis of $V= \RR^4$. $M$ is
equipped with the $\SL(2,\CC)$-cohomogeneity one action given by left
multiplication on the first factor; we will show that the Riemannian
metric $g_t + dt^2$ on the regular set $\SL(2,\CC)\times I$ extends to
a smooth $\SL(2,\CC)$-invariant Riemannian metric $g$ on $M$. 

This extension problem for cohomogeneity one actions (or more
generally for polar actions) has been considered by several authors,
see e.g.~\cite{EschenburgWang}, \cite{Mendes}, \cite{Wang},
\cite{Reidegeld}. Note that since we considered an arbitrary $\SU(2)$-invariant left-invariant initial value of the Hitchin flow, the Bryant-Salamon metric on the spin bundle of hyperbolic 3- space described in \cite[Section 3]{BryantSalamon} is included in our class of metrics. This shows that there exists a metric in this class that extends smoothly to M. It follows that our Riemannian metrics with $b_1 \neq 0$ extend smoothly to M and are homothetic to the Bryant-Salamon metric. On the final two pages of this paper we nevertheless verify explicitly, without using the result of \cite{BryantSalamon}, that the metric extends. 

We define $\mfp$ as the span of $e_4,e_5,e_6$ in $\msl(2,\CC)$ and note that this space is invariant under the adjoint action of $\su(2)$. We consider $U$ as embedded in $M$ via $v\mapsto [e,v]\in \SL(2,\CC)\times_{\SU(2)}U$; then the tangent bundle $TM$, restricted to $U$, is the trivial $\SU(2)$-bundle
\[
TM|_U = U\times (V\times \mfp),
\]
see \cite{EschenburgWang}, p.~112. An $\SL(2,\CC)$-invariant Riemannian metric on $M$ is thus the same as a smooth $\SU(2)$-equivariant map
\[
U\longrightarrow S^2((V\times \mfp)^*).
\]

We consider the curve $\gamma:(0,\epsilon)\rightarrow \SL(2,\CC)\times (a,b)\hookrightarrow M$, $\gamma(t):=(e,b-t)=[e,tv_1]$ and consider $(-\epsilon,\epsilon)$ as a subset of $M$ via $s\mapsto [e,sv_1]$. Then the fundamental vector fields of $e_i\in \su(2)$, $i=1,2,3$, restricted to $\gamma$, are given by $\frac{tv_{i+1}}{2}$ and the metric along $\gamma$ is given by 
\[
t\mapsto v^1\otimes v^1 +\frac{4\cdot x(b-t)}{t^2 (4-3 x(b-t))^{\frac{1}{3}}}\left(\sum_{i=2}^4 v^i\otimes v^i\right)+ (4-3 x(b-t))^{\frac{2}{3}}\left(\sum_{i=4}^6 e^i\otimes e^i\right).
\]
We note that the metric along $\gamma$ takes values in the $\SU(2)$-invariant subspace $S^2(V^*)\oplus S^2(\mfp^*)$, so we can consider the extension problem separately. First we want to show that the metric extends from $(0,\epsilon)$ to $(-\epsilon,\epsilon)$. As we want to obtain an invariant metric, this extension has to be equivariant under the subgroup of $\SU(2)$ which leaves invariant the line $\RR\cdot v_1\subset V$, which is $\ZZ_2=\{\pm \id\}\subset \SU(2)$. The group $\ZZ_2$ acts on $V$ by $v\mapsto -v$, and trivially on $\mfp$ (recall that $\SU(2)$ acts on $\mfp$ by the adjoint action). Hence, it acts trivially on $S^2(V^*)\oplus S^2(\mfp^*)$, but by $t\mapsto -t$ on $(-\epsilon,\epsilon)$; in other words, we have to show the following lemma:

\begin{lemma}
The coefficient functions of the metric along $\gamma$ can be extended to smooth even functions on $(-\epsilon,\epsilon)$.
\end{lemma}

\begin{proof} We have to show that all the odd derivatives of the coefficient functions vanish in the limit $t=0$. We first argue that the function $t\mapsto x(b-t)$ has this property.
For that we argue by induction, using Equation \eqref{eq:xdot}, that the odd derivatives of $x$ have the form
\begin{equation}\label{eq:oddderivativesexample}
x^{(2k+1)} = \sum_{i=0}^k c_{k,i} \cdot x^{\frac{2i+1}{2}} \cdot (4-3x)^{\alpha_{k,i}}
\end{equation}
for all $k\geq 0$, and that the even derivatives have the form
\begin{equation}\label{eq:evenderivativesexample}
x^{(2k)} = \sum_{i=0}^k d_{k,i}\cdot x^i\cdot (4-3x)^{\beta_{k,i}}
\end{equation}
for all $k\geq 1$, where the $c_{k,i}, d_{k,i}, \alpha_{k,i}$ and $\beta_{k,i}$ are rational constants.
To see this, we calculate that the derivative of a summand of the form $x^{\frac{2i+1}{2}} \cdot (4-3x)^{\alpha_{k,i}}$ in \eqref{eq:oddderivativesexample} reads
\begin{align*}
&\frac{2i+1}{2}\cdot x^{\frac{2i-1}{2}} \cdot \dot{x}\cdot (4-3x)^{\alpha_{k,i}} -3 \alpha_{k,i}\cdot x^{\frac{2i+1}{2}} (4-3x)^{\alpha_{k,i}-1} \dot{x} \\
&= -\frac{2i+1}{2} \cdot x^i\cdot (4-3x)^{\alpha_{k,i} + \frac{1}{6}} - 3\alpha_{k,i} \cdot x^{i+1} (4-3x)^{\alpha_{k,i}+\frac{5}{6}}
\end{align*}
which is the sum of two summands of the form that appear in \eqref{eq:evenderivativesexample}. Similarly, the derivative of a summand in \eqref{eq:evenderivativesexample} results in the sum of one (for $i=0$) or two summands of the form in \eqref{eq:oddderivativesexample}.

Because $\lim_{t\to b^-} x(t)=0$, Equation \eqref{eq:oddderivativesexample} implies that $t\mapsto x(b-t)$ extends to a smooth even function on $(-\epsilon,\epsilon)$, which is zero for $t=0$. Hence, $t\mapsto (4-3 x(b-t))^{\frac{2}{3}}$ (one of the coefficient functions) and $t\mapsto\tfrac{x(b-t)}{t^2}$ extend to smooth even functions on $(-\epsilon,\epsilon)$. But then also the other coefficient function $t\mapsto \tfrac{4 x(b-t)}{t^2 (4-3 x(b-t))^{\frac{1}{3}}}$ extends to a smooth even function on $(-\epsilon,\epsilon)$.
\end{proof} 

Now it is known that the problem to extend the smooth $\ZZ_2$-equivariant map $(-\epsilon,\epsilon)\to S^2(V^*)\oplus S^2(\mfp^*)$ to a smooth $\SU(2)$-equivariant map defined on $U$ is just the extension problem for the coefficients of the Taylor expansion at $0$, see \cite{EschenburgWang}, Lemma 1.1, or \cite{Mendes}, Chapter 3. 

Let us consider the $V$-part of the metric. As observed above, $\ZZ_2$ acts trivially on $S^2(V)$. Thus, the space of invariant polynomials $\RR[\RR,S^2(V)]^{\ZZ_2}$ is nothing but $\RR[t^2]\otimes S^2(V)$. All spaces of polynomials we consider carry the natural grading by degree of the polynomials. We recall that the Poincar\'e series of a graded vector space $V=\bigoplus_{k\geq 0} V_k$ such that each space $V_k$ is finite-dimensional, is the formal series $\sum_{k\geq 0} t^k \dim V_k$. Thus, the space $\RR[\RR,S^2(V)]^{\ZZ_2}$ has the Poincar\'e series 
\[
P_t(\RR[\RR,S^2(V)]^{\ZZ_2}) = 10+10t^2+10t^4+\cdots
\]

To compute the Poincar\'e series of $\RR[V,S^2(V)]^{\SU(2)}\subseteq \RR[\RR,S^2(V)]^{\ZZ_2}$, we complexify the situation: Denoting by $W_n$, $n=0,1,2,\ldots$, the irreducible complex $\SU(2)$-representation $S^n(\CC^2)$, we have $V^\CC = W_1\oplus W_1$; we thus have to find, for each degree $k$, the trivial summands in $S^k(W_1\oplus W_1)\otimes S^2(W_1\oplus W_1)$. Using the Clebsch-Gordon formula
\[
W_n\otimes W_m = W_{n+m}\oplus W_{n+m-2}\oplus \cdots\oplus W_{|n-m|},
\]
see \cite{FultonHarris}, Exercise 11.11, we compute
\begin{align*}
S^k(W_1\oplus W_1)\otimes S^2(W_1\oplus W_1) =& \left(\bigoplus_{l=0}^k W_l\otimes W_{k-l}\right)\otimes (3W_2\oplus W_0)\\
= & \left(\bigoplus_{l=0}^k W_k\oplus W_{k-2} \oplus\cdots\oplus W_{|k-2l|}\right)\otimes (3W_2\oplus W_0),
\end{align*}
from which we deduce that the desired Poincar\'e series is
\[
P_t(\RR[V,S^2(V)]^{\SU(2)}) = 1 + 10t^2 + 10 t^4 + \cdots
\]
This means that the only obstruction for extending the $V$-part of the metric is that $\lim_{t\to 0}  \frac{4x(b-t)}{t^2 (4-3 x(b-t))^{\frac{1}{3}}}=1$, which is easily checked using \eqref{eq:xdot}.

A smooth $\SU(2)$-equivariant extension of the $\mfp$-part of the metric to $U$ is given by $v\mapsto (4-3 x(b-|v|))^{\frac{2}{3}} \left(\sum_{i=4}^6 e^i\otimes e^i\right)$ as $t\mapsto (4-3 x(b-t))^{\frac{2}{3}}$ is a smooth even function and $\sum_{i=4}^6 e^i\otimes e^i$ is $\SU(2)$-invariant. This shows that the metric $g_t+dt^2$ extends smoothly from $\SL(2,\CC)\times (a,b)$ to $M$. 
\end{proof}
\begin{rem}
In fact, for $b_1=0$, not only the Riemannian metric extends from $\SL(2,\CC)\times (a,b)$ to $M$ but also the parallel $\G_2$-structure $\omega_t\wedge dt+\rho_t$. To see this note that, as the holonomy of $(M,g)$ is equal to $\G_2$, $M$ possesses a parallel $\G_2$-structure $\varphi$ which induces exactly the metric $g$. Now the space of $\G_2$-invariant three-forms on a seven-dimensional vector space is one-dimensional by \cite{Bryant}. So $\varphi$ restricted to $\SL(2,\CC)\times (a,b)$ equals $\omega_t\wedge dt+\rho_t$ up to a non-zero scalar multiple. Hence, a non-zero scalar multiple of $\varphi$ is the extension of $\omega_t\wedge dt+\rho_t$ to $M$.
\end{rem}


\begin{thebibliography}{10}
\bibitem{AK} D.~V.~Alekseevski\u{\i}, B.~N.~Kimel'fel'd, {\em The structure of homogeneous Riemannian spaces
with zero Ricci curvature}, Funct.~Anal.~Appl.~{\bf 9} (1975), no.~2, 5--11.
\bibitem{Brand}
A.\ Brandhuber, J.\ Gomis, S.\ Gubser, S.\ Gukov, {\em 
Gauge theory at large $N$ and new $\G_2$ holonomy metrics}, 
Nuclear Phys.\ B {\bf 611} (2001), no.\ 1-3, 179--204. 
\bibitem{Bryant} R.~Bryant, {\em Metrics with Exceptional Holonomy}, Ann.~of~Math.~{\bf 126} (1987), no.~3, 525--576.
\bibitem{BryantSalamon}
R.\ Bryant, S.\ Salamon, {\em On the construction of some complete metrics with exceptional holonomy}, 
Duke Math.\ J.\ {\bf 58} (1989), no.\ 3, 829--850. 
\bibitem{Chong}
Z.\ Chong, M.\ Cveti\v{c}, G.\ Gibbons, H.\ L\"u, C.\ Pope, P.\ Wagner, {\em General metrics of $\G_2$ holonomy and contraction limits}, Nuclear Phys.\ B {\bf 638} (2002), no.\ 3, 459--482.
\bibitem{CleytonSwann} R.\ Cleyton, A.\ Swann, {\em 
Cohomogeneity-one $\G_2$-structures}, J.\ Geom.\ Phys.\ {\bf  44} (2002), no.\ 2-3, 202--220. 
\bibitem{Conti}
D.~Conti, {\em Half-flat nilmanifolds}, Math.~Ann.~{\bf 350} (2011), no.~1, 155--168.
\bibitem{Conti2}
D.~Conti, {\em $\SU(3)$-holonomy metrics from nilpotent Lie groups}, {\texttt{arXiv:1108.2450}}, (2011).
\bibitem{CS} D.~Conti, S.~Salamon, {\emph{Generalised Killing spinors in dimension $5$}}, Trans.~Amer.~Math. Soc.~{\bf 359} (2007), no.~11, 5319--5343.
\bibitem{CLSS}
V.~Cort\'es, T.~Leistner, L.~Sch\"afer, F.~Schulte-Hengesbach, {\em
Half-flat structures and special holonomy}, 
Proc.~Lond.~Math.~Soc.~(3) {\bf 102} (2011), no.~1, 113--158.
\bibitem{CHNP} A.~Corti, M.~Haskins, J.\ Nordstr\"om, T.\ Pacini, \emph{$\G_2$-–manifolds and associative submanifolds
via semi-Fano $3$-–folds},  {\texttt arXiv:math/1207.4470}, (2012).
\bibitem{DK} D.~Deturck, J.~Kazdan, {\em Some regularity theorems in Riemannian geometry}, Ann.\ Sci.\ Ec.\ Norm.\ Super.~{\bf 14} (1981), no.~3, 249--260.
\bibitem{DiazRamos} J.~C.~D\'iaz-Ramos, {\em Proper isometric actions}, \texttt{arXiv:0811.0547}, (2008).
\bibitem{DuistermaatKolk} J.~J.~Duistermaat, J.~A.~Kolk, {\em Lie groups}, Universitext. Springer-Verlag, Berlin, 2000. viii+344 pp.
\bibitem{EschenburgWang} J.-H.~Eschenburg, M.~Y.~Wang, {\em The initial value problem for cohomogeneity one Einstein metrics}, J.~Geom.~Anal.~{\bf 10} (2000), no.~1, 109--137. 
\bibitem{FultonHarris} W.~Fulton, J.~Harris, {\em Representation theory.
A first course}, Graduate Texts in Mathematics, 129. Readings in Mathematics. Springer-Verlag, New York, 1991.
\bibitem{GGK} V.~W.~Guillemin, V.~L.~Ginzburg, Y.~Karshon, {\it Moment maps, cobordisms, and Hamiltonian
group actions}, Mathematical Surveys and Monographs, 96. AMS, Providence, RI,
2002.
\bibitem{Helgason}
S.~Helgason, {\em Differential geometry, Lie groups, and symmetric spaces. 
Corrected reprint of the 1978 original}, Graduate Studies in Mathematics, 34. American Mathematical Society, Providence, RI, 2001.
\bibitem{Hitchin}
N.~Hitchin, {\em Stable forms and special metrics}, Global differential geometry: the mathematical legacy of Alfred Gray (Bilbao, 2000), 70--89.
\bibitem{Hitchin2} N.~Hitchin, {\em The geometry of three-forms in six dimensions}, J.~Differential~Geom.~{\bf 55} (2000), no.~3, 547--576.
\bibitem{KN} S.\ Kobayashi, K.\ Nomizu, {\it Foundations of differential geometry I}, Interscience\ Publ., New York, 1963.
\bibitem{Mendes}
R.~A.~E.~Mendes, {\em Equivariant tensors on polar manifolds}, Publicly accessible Penn Dissertations (2011). Paper 312.
\bibitem{Madsen}
T.\ B.\ Madsen, S.\ Salamon, {\em 
Half-flat structures on $S^3\times S^3$}, 
Ann.\ Global Anal.\ Geom.\ {\bf 44} (2013), no.\ 4, 369--390. 
\bibitem{Mag} L.~Magnin, {\em Sur les alg\`{e}bres de Lie nilpotentes de dimension $\leq 7$}, J.~Geom.~Phys.~ {\bf 3} (1986), no.~1, 119--144.
\bibitem{Palais} R.~S.~Palais, {\em On the existence of slices for actions of non-compact Lie groups}, 
Ann.~of Math.~{\bf 73} (1961), 295--323. 
\bibitem{PSWZ} J.~Patera, R.~ T.~Sharp, P.~Winternitz, H.~Zassenhaus, {\em Invariants of real low dimension Lie algebras}, J.~Mathematical~Phys.~{\bf 17} (1976), no.~6, 986--994.
\bibitem{Reidegeld} F.~Reidegeld, {\em Special cohomogeneity-one metrics with $Q^{1,1,1}$ or $M^{1,1,0}$ as the principal orbit}, J.~Geom.~Phys.~{\bf 60} (2010), no.~9, 1069--1088. 
\bibitem{Reidegeld2} F.\ Reidegeld, {\em
Exceptional holonomy and Einstein metrics constructed from Aloff-Wallach spaces}, Proc.\ Lond.\ Math.\ Soc.\ (3) {\bf 102} (2011), no.\ 6, 1127--1160. 
\bibitem{Stock} S.~Stock, {\em Gauge Deformations and Embedding Theorems for Special Geometries}, {\texttt arXiv:\linebreak math/0909.5549}, (2009).
\bibitem{Wang} M.~Y.~Wang, {\em Einstein metrics from symmetry and bundle constructions}. Surveys in differential geometry: essays on Einstein manifolds, 287--325,
Surv.~Differ.~Geom., VI, Int.~Press, Boston, MA, 1999. 
\bibitem{Warner} F.~W.~Warner, {\it Foundations of differentiable manifolds and Lie groups}, Corrected reprint of the 1971 edition. Graduate Texts in Mathematics, 94. Springer-Verlag, New York-Berlin, 1983.
\end{thebibliography}
\end{document}